\documentclass[a4paper,12pt,final,reqno]{amsart}

\usepackage{amssymb, amsmath, amsthm, fullpage}
\newtheorem{thm}{Theorem}[section]
\newtheorem{prop}[thm]{Proposition}
\newtheorem{lem}[thm]{Lemma}
\newtheorem{rem}[thm]{Remark}

\newtheorem{example}[thm]{Example}

\numberwithin{equation}{section}
\newcommand{\dsum}[2]{{\displaystyle\sum_{#1}^{#2}\,}}
\newcommand{\dprod}[2]{{\displaystyle\prod_{#1}^{#2}\,}}
\newcommand{\dint}[2]{{\displaystyle\int_{#1}^{#2}}\,}

\newcommand{\tsum}[2]{{\textstyle\sum_{#1}^{#2}}\,}

\newcommand{\br}[1]{{\langle{#1}\rangle}}
\newcommand{\pr}[1]{{\left\{{#1}\right\}}}
\newcommand{\iso}{\stackrel{\sim}{\to}}

\newcommand{\norm}[1]{{||{#1}||}}

\title{A direct approach to the 
bispectral problem for the Ruijsenaars-Macdonald $q$-difference operators}
\author{Masatoshi~Noumi and Jun'ichi~Shiraishi}
\address{MN: Department of Mathematics, Kobe University, Rokko, Kobe 657-8501, Japan}
\email{noumi@math.kobe-u.ac.jp}
\address{JS: Graduate School of Mathematical Sciences, University of Tokyo, Komaba, Tokyo 153-8914, Japan}
\email{shiraish@ms.u-tokyo.ac.jp}

\begin{document}

\begin{abstract}
We present a direct approach to the bispectral problem
associated with the Ruijsenaars-Macdonald $q$-difference operators of $GL$ type. 
We give an explicit construction of 
the meromorphic function $\psi_n(x;s|q,t)$ on ${\mathbb T}^n_x \times {\mathbb T}^n_s$, which is 
the solution of the bispectral problem up to a 
certain gauge transformation.
Some basic properties of $\psi_n(x;s|q,t)$ are studied, 
including the structure of the divisor of poles,  the symmetries $\psi_n(x;s|q,t)=\psi_n(s;x|q,t)=\psi_n(x;s|q,q/t)$, and 
the recursive structure 
described in terms of Jackson integrals or $q$-difference operators.  
\end{abstract}
\maketitle
\baselineskip=16pt


\section{Introduction}
In this paper we study the bispectral problem for the Ruijsenaars-Macdonald 
$q$-difference operators of $GL$ type.   
Let $\mathbb{T}^n_x=(\mathbb{C}^\ast)^n$ be the $n$-dimensional 
algebraic torus with canonical coordinates $x=(x_1,\ldots,x_n)$.  
For each $r=0,1,\ldots,n$, we denote by  
$D^{x}_r=D_r(x,T_{q,x}|t)$ the Ruijsenaars-
Macdonald $q$-difference operator of order $r$ 
in $x$ variables with parameter $t\in\mathbb{C}$; it is defined by 
\begin{align}
D_r^{x}=t^{\binom{n}{2}}
\dsum{|I|=r}{}\ 
\dprod{i\in I;\,j\notin I}{}
\dfrac{tx_i-x_j}{x_i-x_j}
\,
\dprod{i\in I}{}
T_{q,x_i}\qquad(r=0,1,\ldots,n), 
\end{align}
summed over all subsets $I\subseteq\pr{1,\ldots,n}$ 
of cardinality $r$, where $T_{q,x_i}$ stands for the 
$q$-shift operator with respect to $x_i$ ($i=1,\ldots,n$).   
Denoting the dual coordinates by $s=(s_1,\ldots,s_n)$, 
we investigate 
the joint bispectral problem 
\begin{align}\label{eq:bispsys}
\begin{array}{ll}
D^{x}_r\,f(x;s)=f(x;s)\,e_r(s)\quad&(r=1,\ldots,n),
\\[8pt]
D^{s}_r\,f(x;s)=f(x;s)\,e_r(x)\quad&(r=1,\ldots,n), 
\end{array}
\end{align}
for an unknown meromorphic function 
$f(x;s)$ on $\mathbb{T}^n_x\times\mathbb{T}^n_{s}$. 
Here, for each $r=0,1\ldots,n$,  
$e_r(s)$ denotes the elementary symmetric polynomial 
in $s$ variables of degree $r$.  
We also write this system \eqref{eq:bispsys} of 
$q$-difference equations as 
\begin{align}\label{eq:bisp}
\begin{array}{ll}
D^{x}(u)\,f(x;s)=f(x;s)\,\dprod{i=1}{n}(1-us_i),\\
D^{s}(u)\,f(x;s)=f(x;s)\,\dprod{i=1}{n}(1-ux_i), 
\end{array}
\end{align}
by using the generating function 
$D^x(u)=\tsum{r=0}{n}\,(-u)^r\,D^x_r$
of the Ruijsenaars-Macdonald operators. 

It is known by Ruijsenaars \cite{R1987} and Macdonald \cite{M1995} 
that the $q$-difference operators $D^x_1,\ldots,D^x_n$ 
commute with each other.  
On the ring $\mathbb{C}[x]^{\mathfrak{S}_n}$ of symmetric 
polynomials in $x$,  
they are diagonalized simultaneously 
by the Macdonald polynomials $P_{\lambda}(x|q,t)$ 
associated with partitions $\lambda=(\lambda_1,\ldots,\lambda_n)$: 
\begin{align}
D^x(u)\,P_{\lambda}(x|q,t)=P_\lambda(x|q,t)\,\dprod{i=1}{n}(1-ut^{n-i}q^{\lambda_i}).  
\end{align}
It is also known that the Macdonald polynomials
$\widetilde{P}_\lambda(x|q,t)={P_\lambda(x|q,t)}/{P_\lambda(t^\delta|q,t)}$
normalized by the values at $t^\delta=(t^{n-1},t^{n-2},\ldots,1)$ 
have the duality property 
\begin{align}
\widetilde{P}_\lambda(t^\delta q^\mu|q,t)
=\widetilde{P}_\mu(t^\delta q^\lambda|q,t)
\end{align}
for any partitions $\lambda$, $\mu$ of length $\le n$.  
This duality implies that the function 
$f(x;s)=\widetilde{P}_{\lambda}(t^\delta q^\mu|q,t)$ 
solves the bispectral problem \eqref{eq:bisp} on the 
discrete set of points $(x;s)=(t^\delta q^\mu,t^\delta q^\lambda)$ 
indexed by partitions. 

The bispectral problem \eqref{eq:bisp} 
for the Ruijsenaars-Macdonald operators in complex variables 
is studied by Cherednik in the context of  difference  Harish-Chandra theory \cite{C2003,C2009}, and by 
van Meer and Stokman \cite{MS2010} in connection with 
the quantum KZ equations.  The existence of meromorphic solutions 
has also been established by their works.   
In this article, we present a direct approach to the bispectral problem
in the region $|x_1| \gg |x_2|\gg\cdots \gg |x_n|$
and $|s_1| \gg |s_2|\gg\cdots \gg |s_n|$, 
based on formal solutions of the 
$q$-difference equations \eqref{eq:bisp} of scalar type.  
In particular we give explicit construction of 
meromorphic solutions, and discuss various properties of 
solutions, including duality and recursive constructions by 
Jackson integral representations and $q$-difference operators.  

First, we establish several fundamental properties concerning the joint
eigenfunction problem with respect to the Ruijsenaars-Macdonald operators acting on $x$
\begin{align}
D^{x}(u)\,f(x;s)=f(x;s)\,\dprod{i=1}{n}(1-us_i),\label{eigen-x}
\end{align}
with  $s=(s_1,\ldots,s_n)$ being a given set of complex variables. Set 
$s_i=t^{n-i}q^{\lambda_i}$ $(1\leq i\leq n)$. 
Let $Q_+={\mathbb N}\alpha_1\oplus \cdots\oplus {\mathbb N}\alpha_{n-1}$ be the positive cone of the root lattice of $A_{n-1}$ with 
simple roots $\alpha_i=\epsilon_i-\epsilon_{i+1}$ ($1\leq i\leq n-1$).
We write $x^\lambda=\prod_i x_i^{\lambda_i}$,
$x^{-\mu}=x_1^{-\mu_1}\cdots x_n^{-\mu_n}$ for each $\mu=\sum_i \mu_i \epsilon_i \in Q_+$, and 
$ {\mathbb C}(s^{-Q_+})[[x^{-Q_+}]]= {\mathbb C}(s_2/s_1,\ldots,s_n/s_{n-1})[[x_2/x_1,\ldots,x_n/x_{n-1}]]$ for short.
Let $f(x;s)$ be a formal series solutions for the joint
eigenfunction problem (\ref{eigen-x})
given by  $f(x;s)=x^\lambda \varphi(x;s)$, $\varphi(x;s)=\sum_{\mu\in Q_+} x^{-\mu} \varphi_\mu(s)\in {\mathbb C}(s^{-Q_+})[[x^{-Q_+}]]$
having the initial term $\varphi_0(s)=1$.
\begin{thm}[Theorem \ref{thm:FR}]\label{thm-1}
$(1)$ Such a series $\varphi(x;s)$ exists uniquely. $(2)$
For each $\mu\in Q_+$ the coefficient $\varphi_\mu(s)\in {\mathbb C}(s^{-Q_+})$ is regular at the origin $(s_2/s_1,\ldots,s_n/s_{n-1})=0$. 
$(3)$ Each $\varphi_\mu(s)$ has at most simple poles along the divisors 
$q^{k+1}s_j/s_i=1$ $(1\leq i<j\leq n;k=0,1,2,\ldots.)$. 
\end{thm}

Set  $\psi(x;s)$ via the gauge transformation 
\begin{align}
\varphi(x;s)=\prod_{1\leq i<j\leq n}{(qx_j/x_i;q)_\infty \over (qx_j/tx_i;q)_\infty} \psi(x;s),
\end{align}
where we used the standard notation
for the  $q$-shifted factorial
$(z;q)_\infty=\prod_{k=0}^{\infty}\,(1-q^k z)$ assuming that $|q|<1$. Then (\ref{eigen-x}) is recast as 
\begin{align}
L^{(x;s)}(u) \psi(x;s)= \psi(x;s)\prod_{i=1}^n(1-u s_i),
\end{align}
where 
\begin{align}
&L^{(x;s)}(u)
=
\dsum{I\subseteq\pr{1,\ldots,n}}{}(-u)^{|I|}s^{\epsilon_I}B_{I}(x)\,T_{q,x}^{\,\epsilon_I}\qquad (\epsilon_I=\sum_{i\in I}\epsilon_i),
\nonumber\\
&B_{I}(x)=
\dprod{i<j;\, i\notin I, j\in I}{}\dfrac{1-tx_j/x_i}{1-x_j/x_i}
\dfrac{1-qx_j/tx_i}{1-qx_j/x_i}
\quad(I\subseteq \pr{1,\ldots,n}).
\end{align}
Since $L^{(x;s)}(u)$ is invariant under the exchange $t \leftrightarrow q/t$, we have
\begin{thm}[Theorem \ref{thm:FRpsi}]
The series $\psi(x;s)$ is invariant under the change of parameters $t \leftrightarrow q/t$.
\end{thm}

Now we turn to an explicit formula for the eigenfunction of (\ref{eigen-x}).
For each $n=1,2,\ldots$, 
we denote by $M_n$ the set of all strictly upper triangular 
$n\times n$ matrices with nonnegative integer coefficients:
$
M_n=\{ \theta=(\theta_{ij})_{i,j=1}^{n}\in\mbox{Mat}(n;\mathbb{N})\ |\ \  
\theta_{ij}=0\ \ (1\le j\le i\le n) \}. 
$
For each $\theta\in M_n$, we define a rational function 
$c_n(\theta;s|q,t)$ in the variables $s=(s_1,\ldots,s_n)$ by
\begin{align}\label{def-cn}
c_n(\theta;s|q,t)
&=
\dprod{k=2}{n}
\dprod{1\le i<j\le k}{}
\dfrac{(q^{\sum_{a>k}(\theta_{i,a}-\theta_{j,a})}ts_j/s_i;q)_{\theta_{i,k}}}
{(q^{\sum_{a>k}(\theta_{i,a}-\theta_{j,a})}qs_j/s_i;q)_{\theta_{i,k}}}
\nonumber\\
&\quad\cdot
\dprod{k=2}{n}
\dprod{1\le i\le j<k}{}
\dfrac{(q^{-\theta_{j,k}+\sum_{a>k}(\theta_{i,a}-\theta_{j,a})}qs_j/ts_i;q)_{\theta_{i,k}}}
{(q^{-\theta_{j,k}+\sum_{a>k}(\theta_{i,a}-\theta_{j,a})}s_j/s_i;q)_{\theta_{i,k}}}, 
\end{align}
where 
$(z;q)_k=(z;q)_\infty/(q^kz;q)_\infty$ $(k\in\mathbb{Z})$, and 
$\sum_{a>k}=\sum_{a=k+1}^{n}$. 
Set 
\begin{align}
p_n(x;s|q,t)&=\sum_{\theta\in M_n}c_n(\theta;s|q,t) \prod_{1\leq i<j\leq n} (x_j/x_i)^{\theta_{ij}}\in {\mathbb C}(s^{-Q_+})[[x^{-Q_+}]].
\end{align}

\begin{thm}[Theorem \ref{thm:eigenDfn}]
We have the eigenfunction equation in $x$
\begin{align}
D^x(u) x^\lambda p_n(x;s|q,t)=x^\lambda p_n(x;s|q,t)\prod_{i=1}^n (1-u s_i).
\end{align}
\end{thm}

In Section 3, this will be proved by using the property of $c_n(\theta;s|q,t)$ and the theory of Macdonald polynomials.
We note that an alternative proof which does not rely on the Macdonald theory will be 
presented in Sections 4 and 5 based on a recursive use of certain $q$-difference operators.

Now we embark on the study of the bispectral problem (\ref{eq:bispsys}).
Let $e_n(x;s)$ be a (possibly multi-valued) meromorphic function in $x,s$ such that 
\begin{align}
T_{q,x_i} e_n(x;s)=e_n(x;s) s_i t^{-n+i},\qquad 
T_{q,s_i} e_n(x;s)=e_n(x;s) x_i t^{-n+i} \qquad (1=1,\ldots,n),
\end{align}
and $e_n(x;s)=e_n(s,x)$.
Set 
\begin{align}
\varphi_n(x;s|q,t)&=
\prod_{1\leq i<j\leq n}  {(qs_j/s_i;q)_\infty \over (qs_j/t s_i;q)_\infty}p_n(x;s|q,t),\\
\psi_n(x;s|q,t)&=
\prod_{1\leq i<j\leq n}  {(qx_j/tx_i;q)_\infty \over (qx_j/x_i;q)_\infty}p_n(x;s|q,t)\nonumber\\
&=
\prod_{1\leq i<j\leq n}  {(qx_j/tx_i;q)_\infty \over (qx_j/x_i;q)_\infty}
 {(qs_j/ts_i;q)_\infty \over (qs_j/s_i;q)_\infty}
\varphi_n(x;s|q,t),\\
f_n(x;s|q,t)&=e_n(x;s) \varphi_n(x;s|q,t)\in e_n(x;s) {\mathbb C}[[s^{-Q_+}]][[x^{-Q_+}]].
\end{align}
By a recursive use of  certain Jackson integral transformations studied in \cite{MN1997}, we have
\begin{thm}[Theorem \ref{bispect-thm}]\label{bis}
The series $f_n(x;s|q,t)$ is a formal solution of the bispectral problem (\ref{eq:bispsys})
associated with the Ruijsenaars-Macdonald operators.
We have the symmetries
\begin{align}
\varphi_n(x;s|q,t)&=\varphi_n(s;x|q,t),\\
\varphi_n(x;s|q,t)&=\prod_{1\leq i<j\leq n}
{(t x_j/x_i;q)_\infty  \over (q x_j/t x_i;q)_\infty }
{(t s_j/s_i;q)_\infty  \over (q s_j/t s_i;q)_\infty }\varphi_n(x;s|q,q/t),\\
\psi_n(x;s|q,t)&=\psi_n(s;x|q,t),\\
\psi_n(x;s|q,t)&=\psi_n(x;s|q,q/t).
\end{align}
\end{thm}

We can recast the recurrence of the Jackson integrals by 
that of a certain $q$-difference operators. 
Note that such a class of $q$-difference operators, called the Ruijsenaars-Macdonald operators of 
{\it row type}, is
studied in \cite{NS2012}.
Set
\begin{align}
K^{(x;s|q,t)}(u)
=\dsum{\nu\in\mathbb{N}^n}{}
\dprod{i=1}{n}(u/s_i)^{\nu_i}
\dprod{i=1}{n}
\dfrac{(t;q)_{\nu_i}}{(q;q)_{\nu_i}}
\dprod{1\le i<j\le n}{}
\dfrac{(tx_j/x_i;q)_{\nu_i}}{(qx_j/x_i;q)_{\nu_i}}
\dfrac{(q^{-\nu_j+1}x_j/t x_i;q)_{\nu_i}}{(q^{-\nu_j}x_j/x_i;q)_{\nu_i}}\ 
T_{q,x}^{-\nu}. 
\end{align}

\begin{thm}[Theorem \ref{Kxs}]
We have the recurrence relations
\begin{align}
\psi_{n+1}(s;x|q,t)
&=\prod_{i=1}^n
{(tx_{n+1}/x_i;q)_\infty \over (qx_{n+1}/x_i;q)_\infty }
{(qs_{n+1}/ts_i;q)_\infty \over (qs_{n+1}/s_i;q)_\infty }\\
&\cdot K^{(x;s|q,t)}(q s_{n+1}/t)K^{(s;x|q,q/t)}(t x_{n+1}) \psi_n (s;x|q,t)\qquad (n=1,2,\ldots).\nonumber
\end{align}
\end{thm}
This provides us with an alternative proof of Theorem \ref{bis} which does not depend on the theory of Macdonald polynomials.

Finally, let $e_n(x;s)\varphi_n(x;s|q,t)$ be our formal solution of the bispectral problem for the Ruijsenaars-Macdonald operators, and set 
\begin{align}
F_n(x;s|q,t)=\prod_{1\leq i<j\leq n} (q x_j/t x_i)_\infty (q s_j/t s_i)_\infty \, \varphi_n(x;s|q,t).
\end{align}
\begin{thm}[Theorem \ref{holo}]
The series $F_n(x;s|q,t)$ represents a holomorphic function on ${\mathbb C}_z^{n-1}\times {\mathbb C}_w^{n-1}$
in the variables $(z;w)=(z_1,\ldots,z_{n-1};w_1,\ldots,w_{n-1})$
with $z_i=x_{i+1}/x_i,w_i=s_{i+1}/s_i$ $(i=1,\ldots,n-1)$, depending holomorphically on $t\in {\mathbb C}^*$ such that 
$F_n(x;s|q,t)=F_n(s;x|q,t)=F_n(x;s|q,q/t)$.
\end{thm}

As an application of  our results, we have a summation formula which can be regarded as an infinite series version of the
principal substitution of the Macdonald polynomial $P_\lambda(t^\delta|q,t)$.
\begin{thm}[Theorem \ref{pric-sp}]
Let $|t|>|q|^{-(n-2)}$.  
We have
\begin{align}
\dsum{\theta\in M_n}{}c_n(\theta;s|q,t)\,t^{\sum_{i<j}(i-j)\theta_{ij}}
=
\dprod{i=1}{n}\dfrac{(q/t;q)_\infty}{(q/t^i;q)_\infty}
\dprod{1\le i<j\le n}{}
\dfrac{(qs_j/ts_i;q)_\infty}{(qs_j/s_i;q)_\infty}.  
\end{align}
\end{thm}

In view of the explicit construction of $\varphi_n(x;s|q,t)$ as above,
one observes that the properties stated in Theorems \ref{thm-1} and \ref{bis} are hidden and not evident at all.
It seems a challenging problem to have an alternative expression of $\varphi_n(x;s|q,t)$ in which 
such properties are manifest.
At least the cases $n=2$ and $3$ can be worked out. 
The case $n=2$ is easy and given in (\ref{eq:n=2final}).
\begin{thm}[Theorem \ref{n=3}]
We have 
\begin{align}
&\varphi_3(x,s|q,t)=%
\sum_{k\geq 0}
{(q/t;q)_k(q/t;q)_k \over 
(q;q)_k(t;q)_k} 
( qx_3s_3/x_1s_1)^k \\
&\quad \times
\prod_{1\leq i<j\leq 3}
{(t;q)_\infty (q x_js_j/x_is_i;q)_\infty \over (q x_j/tx_i;q)_\infty (q s_j/ts_i;q)_\infty}\,\,
{}_2\phi_1\left[
{qx_j/tx_i,qs_j/ts_i \atop q x_js_j/x_is_i};q,q^k t\right] \qquad (|t|<1),\nonumber
\end{align}
which manifestly shows the 
regularity in Theorem \ref{thm-1}, and the 
duality $\varphi_3(x,s|q,t)=\varphi_3(s,x|q,t)$.
\end{thm}

As yet another application, we will revisit the problem considered in \cite{S2006}, and 
examine some properties of a certain family of integral transformations called $I(\alpha)$  ($\alpha\in {\mathbb C}$) (see Section 8).
We claim that  we have the commutativity $[D^x(u),I(\alpha)]=0$ and $[I(\alpha),I(\beta)]=0$.

Organization of the paper is as follows.
In Section 2, we establish several fundamental properties of formal solutions 
for the joint eigenfunction problem associated with the Ruijsenaars-Macdonald operators in the variables $x$.
In Section 3, we study the properties of the explicit formulas for the eigenfunction. 
In Section 4, the recursive structure in the explicit formula is investigated from 
the point of view of the Jackson integrals, which enables us to study the eigenfunction problem associated with the
Ruijsenaars-Macdonald operators in the variables $s$.
In Section 5, the recursive structure is recast in 
yet another form based on a recursive application of certain $q$-difference operators,
thereby establishing the bispectral eigenfunction problem associated with the
Ruijsenaars-Macdonald operators in the pair of variables $x$ and $s$, 
without relying on the theory of Macdonald polynomials.
Section 6 is devoted to the study of the convergence of the formal series solutions.
In Section 7, we treat the case $n=3$ and obtain some formulas for the eigenfunctions in which 
the duality can be seen manifestly. As an application of our result, we study
the family of integral transformations $I(\alpha)$ ($\alpha\in {\mathbb C}$) in Section 8.


\par\medskip
Throughout this paper, we fix nonzero constants $q, t\in\mathbb{C}^\ast$ 
with the assumption $|q|<1$. 
We use the standard notation
of $q$-shifted factorials
\begin{align}
(z;q)_\infty=\dprod{k=0}{\infty}\,(1-q^k z),
\qquad
(z;q)_k=\dfrac{(z;q)_\infty}{(q^kz;q)_\infty}\quad(k\in\mathbb{Z}),
\end{align}
theta function
$\theta(z;q)=(z;q)_\infty (q/z;q)_\infty (q;q)_\infty$,
and Jackson integrals
\begin{align}
&\dint{0}{a}f(z)\dfrac{d_qz}{z}=(1-q)\dsum{k=0}{\infty}f(q^ka),
\quad
\dint{a}{\infty}f(z)\dfrac{d_qz}{z}=(1-q)\dsum{k=0}{\infty}f(q^{-k-1}a), 
\nonumber\\
&\dint{0}{a\infty}f(z)\dfrac{d_qz}{z}=
\dint{0}{a}f(z)\dfrac{d_qz}{z}+
\dint{a}{\infty}f(z)\dfrac{d_qz}{z}=
(1-q)\dsum{k=-\infty}{\infty}f(q^ka).  
\end{align}
The basic hypergeometric series ${}_{r+1}\phi_{r}(a_1,a_2,\ldots,a_{r+1};b_1,\dots,b_{r};q,z)$ $(r=0,1,\ldots)$ is 
defined by 
\begin{align}
{}_{r+1}\phi_r\left[ {a_1,a_2,\ldots,a_{r+1}\atop b_1,\dots,b_{r}};q,z\right]=
\sum_{n=0}^\infty {(a_1,q)_n (a_2,q)_n \cdots (a_{r+1},q)_n \over 
(q,q)_n (b_1,q)_n \cdots (b_{r},q)_n} z^n.
\end{align}
We also use the notation for the very well-poised series as
\begin{align}
&{}_{r+1}W_r(a_1;a_4,a_5,\ldots,a_{r+1};q,z)\\
&=
{}_{r+1}\phi_r\left[ {a_1,q a_1^{1/2}-q a_1^{1/2},a_4,\ldots,a_{r+1}\atop
a_1^{1/2},-a_1^{1/2},q a_1/a_4,\ldots,q a_1/a_{r+1}};q,z\right].\nonumber
\end{align}

The authors wish to thank Y. Komori and M. Lassalle for stimulating discussion.
Research of J.S. is supported by the Grant-in-Aid for Scientific Research C-24540206.

\newpage
\tableofcontents
\newpage


\section{Formal solutions of the eigenfunction equation}

In this section we investigate {\em formal solutions} 
$f(x;s)$ of the eigenfunction equation
\begin{align}\label{eq:eigenDf}
D^{x}(u)\,f(x;s)=f(x;s)\,\dprod{i=1}{n}(1-us_i)
\end{align}
in the variables $x=(x_1,\ldots,x_n)$, and 
prove the unique existence of a formal solution 
$f(x;s)$ with a given leading coefficient.
We also establish certain regularity of coefficients of 
$f(x;s)$ with respect to $s=(s_1,\ldots,s_n)$, and   
symmetry of solutions under the change of 
parameters $t\leftrightarrow q/t$.  

\par\medskip
We first introduce some notations in order to clarify the meaning of 
``formal solutions''.  
Let 
$P=\mathbb{Z}\epsilon_1\oplus\cdots\oplus\mathbb{Z}\epsilon_n$ 
be the free $\mathbb{Z}$-module 
with canonical basis $\pr{\epsilon_1,\ldots,\epsilon_n}$, 
and $\br{\ ,\ }\ :\ P\times P\to\mathbb{Z}$ 
the symmetric bilinear form defined by $\br{\epsilon_i,\epsilon_j}
=\delta_{ij}$ ($i,j\in\pr{1,\ldots,n}$).  
We denote by 
\begin{align}
Q_+=\mathbb{N}\alpha_1\oplus\cdots\oplus
\mathbb{N}\alpha_{n-1}\subseteq P\qquad(\mathbb{N}=\pr{0,1,2,\ldots})
\end{align}
the cone generated by the simple roots
$\alpha_i=\epsilon_{i}-\epsilon_{i+1}$ ($i=1,\ldots,n-1$).  
Under the identification $P\iso\mathbb{Z}^n$, we 
use the multi-index notation of monomials 
$x^{\mu}=x_{1}^{\mu_1}\cdots x_n^{\mu_n}$
and $q$-difference operators 
$T_{q,x}^\mu=T_{q,x_1}^{\mu_1}\cdots T_{q,x_n}^{\mu_n}$ 
for each $\mu=\sum_{i=1}^{n}\mu_i\epsilon_i=(\mu_1,\ldots,\mu_n)\in P$, 
so that 
$T_{q,x}^\mu(x^\nu)=q^{\br{\mu,\nu}}x^{\nu}$ 
for any $\mu,\nu\in P$. 
We use the notations 
\begin{align}
\mathbb{C}[[x^{-Q_+}]]=\mathbb{C}[[x_2/x_1,\ldots,x_n/x_{n-1}]]
\quad
\mbox{and}\quad 
\mathbb{C}(x^{-Q_+})=\mathbb{C}(x_2/x_1,\ldots,x_n/x_{n-1})
\end{align} 
for the ring of formal power series  
and the field of rational functions 
in $(n-1)$ variables $x^{-\alpha_i}=x_{i+1}/x_i$ ($i=1,\ldots,n-1$), 
respectively. 
When we work with these algebras, through the identification 
$z_i=x_{i+1}/x_i$ ($i=1,\ldots,n-1)$, 
we use 
$(x_2/x_1,\ldots,x_n/x_{n-1})$ 
as a conventional notation for the canonical coordinates 
$z=(z_1,\ldots,z_{n-1})$ 
of an $(n-1)$-dimensional affine space 
$\mathbb{C}^{n-1}_z$.  
By this convention, we interpret 
the expression $x_j/x_i$ ($i<j$)  
as representing the monomial $z_iz_{i+1}\cdots z_{j-1}$. 

\subsection{Formal solution with a given leading coefficient}
We consider formal solutions $f(x;s)$ of \eqref{eq:eigenDf} in 
the form
\begin{align}\label{eq:fxphi}
f(x;s)=x^\lambda\varphi(x;s),\qquad
\varphi(x;s)=\dsum{\mu\in Q_+}{} x^{-\mu}\varphi_\mu(s), 
\end{align}
where $\lambda=(\lambda_1,\ldots,\lambda_n)$ denotes the 
complex variables such that $s_i=t^{n-i}q^{\lambda_i}$ 
($i=1,\ldots,n$).  
As far as $q$-difference equations are concerned, 
the power function $x^\lambda=x_1^{\lambda_1}\cdots x_n^{\lambda_n}$ 
above may be replaced 
by any function $e(x;s)$ such that 
$T_{q,x_i}e(x;s)=e(x;s)q^{\lambda_i}=e(x;s)s_i/t^{n-i}$
($i=1,\ldots,n$).  
In the expansion \eqref{eq:fxphi} of $\varphi(x;s)$, 
we assume that all the coefficients $\varphi_{\mu}(s)$  
belong either to the ring 
$\mathbb{C}[[s^{-Q_+}]]$ of formal power series, or to the ring 
$\mathbb{C}(s^{-Q_+})$ of rational functions in 
the variables $(s_2/s_1,\ldots,s_n/s_{n-1})$.  
The coefficient $\varphi_0(s)$ is called the {\em leading coefficient} 
of the solution $f(x;s)$, or the {\em initial value of} 
$\varphi(x;s)$ {\em at the origin} $(x_2/x_1,\ldots,x_n/x_{n-1})=0$, 
depending on the situation.  

\begin{thm}[Unique existence of a formal solution]
\label{thm:FF} 
Let $\lambda=(\lambda_1,\ldots,\lambda_n)$ be 
the complex variables  such that $s_i=t^{n-i}q^{\lambda_i}$ for $i=1,\ldots,n$. 
Then, for any 
$c(s)\in\mathbb{C}[[s^{-Q_+}]]$ $($resp. $\in\mathbb{C}(s^{-Q_+})$$)$, 
there exists a unique formal solution 
\begin{align}
f(x;s)=x^\lambda\varphi(x;s)\in x^\lambda\, \mathbb{C}[[s^{-Q_+}][[x^{-Q_+}]]
\quad(\mbox{resp.} \in x^\lambda\, \mathbb{C}(s^{-Q_+})[[x^{-Q_+}]])
\end{align}
of \eqref{eq:eigenDf} with leading coefficient $\varphi_0(s)=c(s)$.  
\end{thm}

\begin{thm}[Formal solution with rational coefficients]
\label{thm:FR} 
With the variables $\lambda=(\lambda_1,\ldots,\lambda_n)$ 
such that $s_i=t^{n-i}q^{\lambda_i}$ $(i=1,\ldots,n)$, let 
\begin{align}
f(x;s)=x^\lambda\varphi(x;s)\in x^\lambda\,\mathbb{C}(s^{-Q_+})[[x^{-Q_+}]],
\quad
\varphi(x;s)=\dsum{\mu\in Q_+}{}x^{-\mu}\varphi_\mu(s), 
\end{align}
be the formal solution 
of \eqref{eq:eigenDf} with leading coefficient $\varphi_0(s)=1$.  
Then, for each $\mu\in Q_+$, 
the coefficient $\varphi_\mu(s)\in \mathbb{C}(s^{-Q_+})$
is regular at the origin $(s_2/s_1,\ldots,s_n/s_{n-1})=0$, and 
has at most simple poles along the divisors 
\begin{align}
q^{k+1}s_{j}/s_i=1\qquad(1\le i<j\le n;\, k=0,1,2,\ldots).
\end{align}
\end{thm}

In order to prove these theorems, we rewrite the eigenfunction 
equation for $f(x;s)$ to that of $\varphi(x;s)$. 
The $q$-difference operator 
\begin{align}
E^{(x;s)}(u)=x^{-\lambda} D^{x}(u)\, x^\lambda 
\end{align}
defined by conjugation is expressed explicitly as 
\begin{align}
&E^{(x;s)}(u)=\dsum{I\subseteq\pr{1,\ldots,n}}{}
(-u)^{|I|}
s^{\epsilon_I}\,
A_I(x)\, T_{q,x}^{\,\epsilon_I},
\nonumber\\
&A_I(x)=
\dprod{i<j; i\in I, j\notin I}{}
\dfrac{1-x_j/tx_i}{1-x_j/x_i}
\dprod{i<j; i\notin I, j\in I}{}
\dfrac{1-tx_j/x_i}{1-x_j/x_i}
\quad(I\subseteq\pr{1,\ldots,n}), 
\end{align}
where $\epsilon_I=\sum_{i\in I}\epsilon_i$.  
With this $q$-difference operator, the eigenfunction equation for $\varphi(x;s)$ is
given by 
\begin{align}\label{eq:eigenEphi}
E^{(x;s)}(u)\,\varphi(x;s)=\varphi(x;s)\,\dprod{i=1}{n}(1-us_i). 
\end{align}
Note that $E^{(x;s)}(u/s_1)$ acts naturally both on 
$\mathbb{C}[[s^{-Q_+}]][[x^{-Q_+}]]$
and on $\mathbb{C}(s^{-Q_+})[[x^{-Q_+}]]$, 
with the coefficients $A_I(x)$ regarded as elements of $\mathbb{C}[[x^{-Q_+}]]$.  

\subsection{Proof of the unique existence}
In the following arguments, 
it is essential to transform this equation once more, 
into the eigenfunction equation for $\psi(x;s)$ defined by
\begin{align}
\varphi(x;s)=\dprod{1\le i<j\le n}{}\dfrac{(qx_j/x_i;q)_\infty}{(qx_j/tx_i;q)_\infty}\,\psi(x;s),
\quad
\psi(x;s)=\dsum{\mu\in Q_+}{}x^{-\mu}\psi_{\mu}(s). 
\end{align}
Note here that $\varphi(x;s)$ and $\psi(x;s)$ have a common leading coefficient;  namely, 
$\varphi_0(s)=\psi_0(s)$.  
We define the $q$-difference operator $L^{(x;s)}$ by setting 
\begin{align}
L^{(x;s)}(u)=
\dprod{1\le i<j\le n}{}\dfrac{(qx_j/tx_i;q)_\infty}{(qx_j/x_i;q)_\infty}\ 
E^{(x;s)}(u)\dprod{1\le i<j\le n}{}\dfrac{(qx_j/x_i;q)_\infty}{(qx_j/tx_i;q)_\infty}. 
\end{align}
By this conjugation, 
we obtain the $q$-difference operator
\begin{align}\label{eq:BI}
&L^{(x;s)}(u)
=
\dsum{I\subseteq\pr{1,\ldots,n}}{}(-u)^{|I|}s^{\epsilon_I}B_{I}(x)\,T_{q,x}^{\,\epsilon_I},
\nonumber\\
&B_{I}(x)=
\dprod{i<j;\, i\notin I, j\in I}{}\dfrac{1-tx_j/x_i}{1-x_j/x_i}
\dfrac{1-qx_j/tx_i}{1-qx_j/x_i}
\quad(I\subseteq \pr{1,\ldots,n}).
\end{align}
It should be observed that the coefficients $B_I(x)$ of $L^{(x;s)}(u)$ 
depend on $t$ and $q/t$, symmetrically. 
We will investigate below 
the eigenfunction equation
\begin{align}\label{eq:eigenLpsi}
L^{(x;s)}(u)\psi(x;s)=\psi(x;s)\dprod{i=1}{n}(1-us_i),
\quad
\psi(x;s)=\dsum{\nu\in Q_+}{}x^{-\mu}\psi_\mu(s) 
\end{align}
for $\psi(x;s)$.   
Theorems \ref{thm:FF} and \ref{thm:FR} are derived from 
the following theorems concerning the eigenfunction equation 
\eqref{eq:eigenLpsi}.

\begin{thm}\label{thm:FFpsi}
For any $c(s)\in\mathbb{C}[[s^{-Q_+}]]$ 
$($resp. $\in \mathbb{C}(s^{-Q_+})$$)$, 
there exists a unique 
formal solution $\psi(x;s)$ in 
$\mathbb{C}[[s^{-Q_+}]][[x^{-Q_+}]]$ 
$($resp. in $\mathbb{C}(s^{-Q_+})[[x^{-Q_+}]]$$)$
of \eqref{eq:eigenLpsi} 
with leading 
coefficient $\psi_0(s)=c(s)$. 
\end{thm}

\begin{thm}\label{thm:FRpsi}
Let $\psi(x;s)\in\mathbb{C}(s^{-Q_+})[[x^{-Q_+}]]$ be 
the formal solution of \eqref{eq:eigenLpsi} with $\psi_0(s)=1$. 
\newline
$(1)$\ 
For each $\mu\in Q_+$, 
$\psi_{\mu}(s)\in \mathbb{C}(s^{-Q_+})$ is 
regular at the origin 
$(s_2/s_1,\ldots,s_n/s_{n-1})=0$, and has at most simple poles 
along the divisors 
$q^{k+1}s_j/s_i=1\quad(1\le i<j\le n;\ k=0,1,2\ldots)$.  
\newline
$(2)$\ This formal solution $\psi(x;s)$ is invariant under the 
change of parameters $t\leftrightarrow q/t$.  
\end{thm} 

\par\noindent 
We remark that the gauge factor can be expanded as 
\begin{align}
\dprod{1\le i<j\le n}{}\dfrac{(qx_j/x_i;q)_\infty}{(qx_j/tx_i;q)_\infty}
=\dsum{\theta\in M_n}{}\,
\dprod{1\le i<j\le n}{}\,\dfrac{(t;q)_{\theta_{ij}}}{(q;q)_{\theta_{ij}}} (qx_j/tx_i)^{\theta_{ij}}
\in\mathbb{C}[[x^{-Q_+}]]
\end{align}
by the $q$-binomial theorem, 
where $M_n$ denotes the set of all strictly upper triangular $n\times n$ 
matrices $\theta=(\theta_{ij})_{i,j=1}^{n}$ with nonnegative integer coefficients:
\begin{align}
M_n=\pr{\ \theta=(\theta_{ij})_{i,j=1}^{n}\in\mbox{Mat}(n;\mathbb{N})\ |\ \  
\theta_{ij}=0\ \ (1\le j\le i\le n)}. 
\end{align}
Hence the coefficients $\varphi_\mu(s)$ are expressed as $\mathbb{C}$-linear 
combinations of $\psi_\nu(s)$, and {\em vice versa}.  
Theorem \ref{thm:FRpsi}, (1) for $\psi(x;s)$ 
thus implies Theorem \ref{thm:FR} for $\varphi(x;s)$.  
A characteristic feature of $\psi(x;s)$ is 
the symmetry with respect to the change of parameters 
$t\leftrightarrow q/t$;   
it is a consequence from the symmetry of the 
operator $L^{(x;s)}(u)$ by the unique existence of a formal solution 
with leading coefficient 1.  
In the rest of this section, we give a proof of Theorem \ref{thm:FFpsi} 
and Theorem \ref{thm:FRpsi}, (1).  

\par\medskip
By expanding the coefficients $B_I(x)$ in the form
\begin{align}
B_I(x)=\dsum{\mu\in Q_+}{} x^{-\mu} b_{\mu,I}\in \mathbb{C}[[x^{-Q_+}]], 
\end{align}
we rewrite the operator $L^{(x;s)}(u)$ as follows: 
\begin{align}
&L^{(x;s)}(u)=\dsum{\mu\in Q_+}{}x^{-\mu}b_{\mu}(u;s;T_{q,x}),
\nonumber\\
&
b_{\mu}(u;s;T_{q,x})=\dsum{I\subseteq\pr{1,\ldots,n}}{}
(-u)^{|I|}b_{\mu,I}s^{\epsilon_I} T_{q,x}^{\,\epsilon_I}
\quad(\mu\in Q_+).  
\end{align}
Note that the leading term of $L^{(x;s)}(u)$ is given by 
\begin{align}
b_0(u;s;T_{q,x})=\dprod{i=1}{n}(1-us_iT_{q,x_i}), 
\end{align}
since $b_{0,I}=1$ for all $I\subseteq\pr{1,\ldots,n}$, 
and that $b_0(u;s;1)=\dprod{i=1}{n}(1-us_i)$ coincides with the 
eigenvalue in \eqref{eq:eigenLpsi}.  
By expanding the both sides of \eqref{eq:eigenLpsi} as formal power series 
in the variables $(x_2/x_1,\ldots,x_n/x_{n-1})$, we obtain the following 
recurrence relations for the coefficients $\psi_{\mu}(u)$:
\begin{align}\label{eq:recpsi}
(b_0(u;s;q^{-\mu})-b_0(u;s;1))\psi_{\mu}(s)
+\dsum{0<\nu\le \mu}{}b_{\nu}(u;s;q^{-\mu+\nu})\psi_{\mu-\nu}(s)=0
\quad(\mu\in Q_+),
\end{align}
where $\le$ denotes the dominance ordering in $Q_+$.  

In order to show the unique existence of a formal solution with 
a given leading coefficient, we use the first order component 
\begin{align}
L^{(x;s)}_1=\dsum{j=1}{n}s_j B_{\pr{j}}(x)
\,
T_{q,x_j}, \quad
B_{\pr{j}}(x)=
\dprod{i=1}{j-1}
\dfrac{(1-tx_j/x_i)(1-qx_j/tx_i)}{(1-x_j/x_i)(1-qx_j/x_i)},
\end{align}
of $L^{(x;s)}(u)=\dsum{r=0}{n}(-u)^r L^{(x;s)}_r$. 
Note that $B_{\pr{j}}(x)\in\mathbb{C}[[x_2/x_1,\ldots,x_j/x_{j-1}]]$, and hence
$b_{\mu,\pr{j}}=0$ 
unless $\mu\in\mathbb{N}\alpha_1+\cdots+\mathbb{N}\alpha_{j-1}$.   
The eigenfunction equation for $\psi(x;s)$ now takes the form
\begin{align}\label{eq:eigenL1psi}
L_1^{(x;s)} \psi(x;s)=\psi(x;s)\bigg(\dsum{j=1}{n}\,s_j\bigg).  
\end{align}
The corresponding recurrence relations for the coefficients 
$\psi_\mu(s)$ are given by
\begin{align}
\big(b_{0,1}(s;q^{-\mu})-b_{0,1}(s;1)\big)\psi_{\mu}(s)
+\dsum{0<\nu\le \mu}{}b_{\nu,1}(s;q^{-\mu+\nu})\psi_{\mu-\nu}(s)=0
\quad(\mu\in Q_+), 
\end{align} 
where
\begin{align}
&b_{0,1}(s;q^{-\mu})-b_{0,1}(s;1)=\dsum{j=1}{n}s_j(q^{-\mu_j}-1),
\nonumber\\
&
b_{\nu,1}(s;q^{-\mu+\nu})=\dsum{j=1}{n} b_{\nu,\pr{j}} s_j q^{-\mu_j+\nu_j}
\quad(0<\nu\le\mu). 
\end{align}
Noting that $b_{0,1}(s;q^{-\mu})-b_{0,1}(s;1)\ne 0$ for any $\mu>0$
as a polynomial in $s=(s_1,\ldots,s_n)$, we immediately see 
that there exists a unique 
formal solution $\psi(x;s)\in \mathbb{C}(s^{-Q_+})[[x^{-Q_+}]]$ 
of \eqref{eq:eigenL1psi}
with a given leading coefficient $\psi_0(s)=c(s)\in\mathbb{C}(s^{-Q_+})$. 

Supposing that $\mu>0$, choose the index $r$ such that 
\begin{align}
\mu=k_r\alpha_r+\cdots+k_{n-1}\alpha_{n-1}, \quad k_r>0.   
\end{align}
This condition for $r$ is equivalent to saying that $\mu_i=0$ ($i<r$) and 
$\mu_r\ne 0$.  
Note that $\mu_r=k_r$, $\mu_j=-k_{j-1}+k_j$ ($r<j<n$) 
and $\mu_n=-k_{n-1}$.  
In this case, for any $\nu\in Q_+$ with 
$0<\nu\le \mu$, we have 
$\nu\in\mathbb{N}\alpha_r+\cdots+\mathbb{N}\alpha_{n-1}$.  
As we remarked already, $b_{\nu,\pr{j}}=0$ unless 
$\nu\in\mathbb{N}\alpha_1+\ldots+\mathbb{N}\alpha_{j-1}$. 
This implies, $b_{\nu,\pr{j}}=0$ for $j=1,\ldots,r$. 
Hence we have
\begin{align}
&b_{0,1}(s;q^{-\mu})-b_{0,1}(s;1)=
s_r(q^{-\mu_r}-1)+\dsum{j=r+1}{n}s_j(q^{-\mu_j}-1),
\nonumber\\
&
b_{\nu,1}(s;q^{-\mu+\nu})=\dsum{j=r+1}{n} b_{\nu,\pr{j}} s_j q^{-\mu_j+\nu_j}
\quad(0<\nu\le \mu), 
\end{align}
and the recurrence relation for $\psi_\mu(s)$ is expressed as 
\begin{align}\label{eq:recpsireg}
\big((q^{-\mu_r}-1)+\sum_{j>r}(q^{-\mu_j}-1)s_j/s_r\big)
\psi_{\mu}(s)
+
\dsum{0<\nu\le \mu}{}
\big(\dsum{j>r}\,b_{\nu,\pr{j}}q^{-\mu_j+\nu_j}s_j/s_r\big) 
\psi_{\mu-\nu}(s)=0. 
\end{align}
Since the coefficient of $\psi_{\mu}(s)$ is invertible 
in $\mathbb{C}[[s^{-Q_+}]]$ for any $\mu>0$,  
this recurrence relation implies the unique existence of a formal solution 
$\psi(x;s)\in \mathbb{C}[[s^{-Q_+}]][[x^{-Q_+}]]$ 
of \eqref{eq:eigenL1psi} with a given leading coefficient 
$\psi_0(s)=c(s)\in\mathbb{C}[[s^{-Q_+}]]$.  

This formal solution $\psi(x;s)$ of 
\eqref{eq:eigenL1psi}, 
either in $\mathbb{C}[[s^{-Q_+}]][[x^{-Q_+}]]$ 
or in $\mathbb{C}(s^{-Q_+})[[x^{-Q_+}]]$, 
actually solves the 
joint eigenfunction equation \eqref{eq:eigenLpsi} 
with respect to the operator $L^{(x;s)}(u)$. 
Since the operator $L^{(x;s)}(u)$ 
commute with $L^{(x;s)}_1$, 
the formal power series 
$\big(L^{(x;s)}(u/s_1)-\prod_{i=1}^{n}(1-us_i/s_1)\big)\psi(x;s)$ 
is again a formal solution of \eqref{eq:eigenL1psi}.  
However, it must be zero by the uniqueness theorem since 
its leading coefficient is 0.  
This completes the proof of Theorem \ref{thm:FFpsi}.  

\subsection{Regularity of the expansion coefficients}
We now proceed to the proof of Theorem \ref{thm:FRpsi}, (1). 
Let $\psi(x;s)\in\mathbb{C}(s^{-Q_+})[[x^{-Q_+}]]$ be 
the formal solution of \eqref{eq:eigenLpsi} with $\psi_0(s)=1$.  
Then by induction on the dominance ordering, 
the recurrence relations 
\begin{align}\label{eq:recpsireg1}
\psi_{\mu}(s)
=
\dsum{0<\nu\le \mu}{}
\dfrac{\sum_{j>r}\,b_{\nu,\pr{j}}q^{-\mu_j+\nu_j}s_j/s_r}
{(1-q^{-\mu_r})+\sum_{j>r}(1-q^{-\mu_j})s_j/s_r}\,
\psi_{\mu-\nu}(s)
\quad(\mu>0)
\end{align}
imply that $\psi_{\mu}(s)$ is regular at $(s_2/s_1,\ldots,s_n/s_{n-1})=0$. 
In fact, by these recurrence relations  
we can say more about the coefficients $\psi_{\mu}(s)$. 
\begin{lem}
Let $\psi(x;s)\in\mathbb{C}(s^{-Q_+})[[x^{-Q_+}]]$ be 
the formal solution of \eqref{eq:eigenLpsi} with $\psi_0(s)=1$.  
Suppose that 
$\mu\in Q_+$ is expressed as 
$\mu=k_r\alpha_r+\cdots+k_{n-1}\alpha_{n-1}$ 
for some $r=1,\ldots,n$$:$
\begin{align}\label{lem:psireg}
\mu\in \mathbb{N}\alpha_{r}+\cdots+\mathbb{N}\alpha_{n-1}
\quad(r=1,\ldots,n). 
\end{align}
Then the coefficient $\psi_{\nu}(s)$ is a rational function 
in $(s_{r+1}/s_{r},\ldots,s_{n}/s_{n-1})$ regular at the origin 
$(s_{r+1}/s_{r},\ldots,s_{n}/s_{n-1})=0$\,$:$ 
\begin{align}
\psi_{\mu}(s)\in
\mathbb{C}(s_{r+1}/s_{r},\ldots,s_{n}/s_{n-1})
\cap
\mathbb{C}[[s_{r+1}/s_{r},\ldots,s_{n}/s_{n-1}]].
\end{align}
\end{lem}
\qed

What remains is to prove that the coefficients 
$\psi_{\mu}(s)\in\mathbb{C}(s^{-Q_+})$ 
have at most simple poles along the divisors 
$q^{k+1}s_j/s_i=1$ ($1\le i<j\le n; k=0,1,2,\ldots$).  
Since $\psi(x;s)$ is a formal solution of the joint eigenfunction 
equation \eqref{eq:eigenLpsi}, its coefficients satisfy the 
recurrence relations \eqref{eq:recpsi} as well: 
\begin{align}\label{eq:recpsicopy}
(b_0(u;s;q^{-\mu})-b_0(u;s;1))\psi_{\mu}(s)
+\dsum{0<\nu\le \mu}{}b_{\nu}(u;s;q^{-\mu+\nu})\psi_{\mu-\nu}(s)=0
\quad(\mu\in Q_+).  
\end{align}
Suppose that $\mu\in Q_+$ is expressed as
$\mu=k_r\alpha_r+\cdots+k_{n-1}\alpha_{n-1}$ with $k_r>0$.  
In view of this recurrence relation, we investigate 
$b_{\nu}(u;s;\xi)$, $\xi=(\xi_1,\ldots,\xi_n)$, for 
$\nu\in\mathbb{N}\alpha_r+\cdots+\mathbb{N}\alpha_{n-1}$. 
\begin{lem}\label{lem:bnu}
For each $\nu\in \mathbb{N}\alpha_r+\cdots+\mathbb{N}\alpha_{n-1}$, 
the polynomial $b_{\nu}(u;s;\xi)$ is expressed as 
\begin{align}
b_{\nu}(u;s;\xi)=\dprod{i=1}{r-1}(1-us_i\xi_i)\ b^{(r)}_{\nu}(u;s;\xi), 
\end{align}
where $b^{(r)}_{\nu}(u;s;\xi)=b_{\nu}(u;s_r,\ldots,s_n;\xi_r,\ldots,x_n)$ 
stands for the corresponding 
polynomial in the case of 
$(n-r+1)$ variables $(x_r,\ldots,x_n)$ and $(s_r,\ldots,s_n)$.  
\end{lem}
\proof
Recall that these polynomials are determined by the expansion
\begin{align}
\dsum{I\subseteq\pr{1,\ldots,n}}{}(-u)^{|I|}\,s^{\epsilon_I} B_I(x)\,\xi^{\epsilon_I}
=\dsum{\nu\in Q_+}{}x^{-\nu}b_{\nu}(u;s;\xi). 
\end{align} 
In order to pick up $b_{\nu}(u;s;\xi)$ for $\nu\in\mathbb{N}\alpha_r+\cdots+
\mathbb{N}\alpha_{n-1}$, we specialize the $x$ variables by 
setting $x_2/x_1=\cdots=x_r/x_{r-1}=0$.  Since $x_j/x_i=0$ if $i<j$ and $i<r$, 
for each $I\subseteq\pr{1,\ldots,n}$ we have 
\begin{align}\label{eq:BIsp}
B_{I}(x)\big|_{x_{i+1}/x_i=0\ (i=1,\ldots,r-1)}
&=\dprod{r\le i<j\le n;\, i\notin I, j\in I}{}\dfrac{1-tx_j/x_i}{1-x_j/x_i}
\dfrac{1-qx_j/tx_i}{1-qx_j/x_i}
\nonumber\\
&=B_{I\cap \pr{r,\ldots,n}}(x_r,\ldots,x_n), 
\end{align}
which depends only on $I\cap\pr{r,\ldots,n}$.  
Hence, for each $\nu\in \mathbb{N}\alpha_r+\cdots+\mathbb{N}\alpha_{n-1}$, 
the polynomial $b_{\nu}(u;s;\xi)$ can be expressed as 
\begin{align}
b_{\nu}(u;s;\xi)
=
\dsum{I\subseteq\pr{1,\ldots,r-1}}{}
\dsum{J\subseteq\pr{r,\ldots,n}}{}
(-u)^{|I\cup J|}\,b^{(r)}_{\nu,J}\,s^{\epsilon_{I\cup J}}\xi^{\epsilon_{I\cup J}}
=\dprod{i=1}{r-1}(1-us_i\xi_i)\,b^{(r)}_{\nu}(u;s;\xi)
\end{align}
by means of the expansion coefficients $b^{(r)}_{\nu,J}$ 
in the case of $(n-r+1)$ variables. 
\qed
\par\medskip
In the setting of \eqref{eq:recpsicopy}, 
by this lemma we have
\begin{align}
&b_0(u;s;q^{-\mu})-b_0(u;s;1)=\dprod{i=1}{r-1}(1-us_i)
\big(\dprod{j=r}{n}(1-us_jq^{-\mu_j})-\dprod{j=r}{n}(1-us_j)\big), 
\nonumber\\
&b_{\nu}(u;s;q^{-\mu+\nu})=
\dprod{i=1}{r-1}(1-us_i)\,b^{(r)}_{\nu}(u;s;q^{-\mu+\nu}).
\end{align}
Hence we obtain the recurrence relation
\begin{align}\label{eq:recpsicopy2}
\big(\dprod{j=r}{n}(1-us_jq^{-\mu_j})-\dprod{j=r}{n}(1-us_j)\big)
\psi_{\mu}(s)
+\dsum{0<\nu\le \mu}{}b_{\nu}^{(r)}(u;s;q^{-\mu+\nu})\psi_{\mu-\nu}(s)=0.  
\end{align}
Now by setting $u=q^{k_r}/s_r=q^{\mu_r}/s_r$, we can determine 
$\psi_{\mu}(s)$ as follows:
\begin{align}
\dprod{j=r}{n}(1-q^{k_r}s_j/s_r)\big)
\psi_{\mu}(s)
=\dsum{0<\nu\le \mu}{}b_{\nu}^{(r)}(q^{k_r}/s_r;s;q^{-\mu+\nu})\psi_{\mu-\nu}(s), 
\end{align}
namely,
\begin{align}\label{eq:recpsifin}
\psi_{\mu}(s)
&
=\dsum{0<\nu\le \mu}{}
\dfrac{b_{\nu}^{(r)}(q^{k_r}/s_r;s;q^{-\mu+\nu})}
{(1-q^{k_r})\prod_{j=r+1}^{n}(1-q^{k_r}s_j/s_r)}
\psi_{\mu-\nu}(s). 
\end{align}
Suppose that $\nu_r=0$, or 
equivalently $\nu\in\mathbb{N}\alpha_{r+1}+\cdots+
\mathbb{N}\alpha_{n-1}$. 
Then by Lemma \ref{lem:bnu} we have 
\begin{align}
b^{(r)}_\nu(u;s;q^{-\mu+\nu})=(1-us_{r}q^{-k_r})\,b^{(r+1)}(u;s;q^{-\mu+\nu}), 
\end{align}
and hence $b^{(r)}_\nu(q^{k_r}/s_r;s;q^{-\mu+\nu})=0$ 
by the substitution $u=q^{k_r}/s_r$. 
This means that
only such $\nu\in Q_+$ that have positive coefficients on $\alpha_r$ 
contribute nontrivially in the recurrence relation \eqref{eq:recpsifin}.  
In other words, we have the recurrence relation 
\begin{align}\label{eq:recpsifin2}
\psi_{\mu}(s)
&
=\dsum{0\le \nu<\mu;\ l_r<k_r}{}
\dfrac{b_{\mu-\nu}^{(r)}(q^{k_r}/s_r;s;q^{-\nu})}
{(1-q^{k_r})\prod_{j=r+1}^{n}(1-q^{k_r}s_j/s_r)}
\psi_{\nu}(s)
\end{align}
for each $\mu\in Q_+$, $\mu>0$, 
where $k_r$ and $l_r$ stands for the coefficient of $\alpha_r$ in 
the expressions
$\mu=k_r\alpha_r+\cdots+k_{n-1}\alpha_{n-1}$ ($k>0$) and  
$\nu=l_r\alpha_r+\cdots+l_{n-1}\alpha_{n-1}$,
respectively.  
By induction on the dominance ordering, we obtain the following. 
\begin{prop}
Let $\psi(x;s)\in\mathbb{C}(s^{-Q_+})[[x^{-Q_+}]]$ be 
the formal solution of \eqref{eq:eigenLpsi} with $\psi_0(s)=1$.  
Suppose that 
$\mu\in Q_+$ is expressed as 
$\mu=k_r\alpha_r+\cdots+k_{n-1}\alpha_{n-1}$ 
$(r=1,\ldots,n)$. 
Then the coefficient $\psi_{\mu}(s)$ is expressed as 
\begin{align}
\psi_{\mu}(s)
=\dfrac{a_\nu(s)}
{\prod_{r\le i<j\le n}{}(qs_j/s_i;q)_{k_i}}
\end{align}
for some polynomial $a_{\mu}(s)$ in the variables 
$(s_{r+1}/s_r,\ldots,s_n/s_{n-1})$. 
\end{prop}
\qed

\noindent 
This proposition is in fact a refinement of Theorem \ref{thm:FRpsi}, (1).  

\begin{rem}\rm
Let us denote by $\psi_n(x;s|q,t)$ the unique formal solution of (2.14)  with leading coefficient 1.
Express the eigenfunction equation \eqref{eq:eigenLpsi} in the form
\begin{align}
\dsum{I\subset\pr{1,\ldots,n}}{}(-u)^{|I|}s^{\epsilon_I}B_{I}(x)
\psi_n(q^{\epsilon_I}x;s|q,t)
=\psi_n(x;s|q,t)\,\dprod{i=1}{n}(1-us_i). 
\end{align}
and specialize this formula by $x_2/x_1=\cdots=x_{r}/x_{r-1}=0$. 
From \eqref{eq:BIsp} it turns out that the formal power series
$\psi_{n}(x;s|q,t)\big|_{x_{i+1}/x_i=0\ (i=1,\ldots,r-1)}$ 
in $(x_r/x_{r-1},\ldots,x_n/x_{n-1})$ satisfy the eigenfunction 
equation \eqref{eq:eigenLpsi} in $(n-r+1)$ variables 
$(x_r,\ldots,x_n)$ and $(s_r,\ldots,s_n)$.  Hence 
by the uniqueness of formal solution with a given leading coefficient, 
we have 
\begin{align}
\psi_{n}(x;s|q,t)\big|_{x_{i+1}/x_i=0\ (i=1,\ldots,r-1)}=
\psi_{n-r+1}(x_r,\ldots,x_n;s_r,\ldots,s_n|q,t).
\end{align}
\end{rem}


\section{Explicit formulas for formal eigenfunctions}

In this section we construct an explicit formal solution of the 
joint eigenfunction equation \eqref{eq:eigenDf} in the form 
\begin{align}
f(x;s)=x^\lambda \varphi(x;s),\quad 
\varphi(x;s)\in \mathbb{C}[[s^{-Q_+}]][[x^{-Q_+}]]
\end{align}
with the complex parameters $\lambda$ such that $s=t^\delta q^\lambda$. 

\subsection{Formal power series $p_n(x;s|q,t)$ and $\varphi_n(x;s|q,t)$}
We recall some notations in \cite{S2005}.
For each $n=1,2,\ldots$, 
we denote by $M_n$ the set of all strictly upper triangular 
$n\times n$ matrices with nonnegative integer coefficients:
\begin{align}
M_n=\pr{\ \theta=(\theta_{ij})_{i,j=1}^{n}\in\mbox{Mat}(n;\mathbb{N})\ |\ \  
\theta_{ij}=0\ \ (1\le j\le i\le n)}. 
\end{align}
For each $\theta\in M_n$, we define a rational function 
$c_n(\theta;s|q,t)$ in the variables $s=(s_1,\ldots,s_n)$ by
\begin{align}\label{eq:defcn}
c_n(\theta;s|q,t)
&=
\dprod{k=2}{n}
\dprod{1\le i<j\le k}{}
\dfrac{(q^{\sum_{a>k}(\theta_{i,a}-\theta_{j,a})}ts_j/s_i;q)_{\theta_{i,k}}}
{(q^{\sum_{a>k}(\theta_{i,a}-\theta_{j,a})}qs_j/s_i;q)_{\theta_{i,k}}}
\nonumber\\
&\quad\cdot
\dprod{k=2}{n}
\dprod{1\le i\le j<k}{}
\dfrac{(q^{-\theta_{j,k}+\sum_{a>k}(\theta_{i,a}-\theta_{j,a})}qs_j/ts_i;q)_{\theta_{i,k}}}
{(q^{-\theta_{j,k}+\sum_{a>k}(\theta_{i,a}-\theta_{j,a})}s_j/s_i;q)_{\theta_{i,k}}}, 
\end{align}
where $\sum_{a>k}=\sum_{a=k+1}^{n}$. 
Note that, for each $\theta\in M_n$, 
$c_n(\theta;s|q,t)$ is in fact a rational function in 
$(s_2/s_1,\ldots,s_n/s_{n-1})$, manifestly regular 
at the origin $(s_2/s_1,\ldots,s_n/s_{n-1})=0$; it 
can be regarded as an element 
of $\mathbb{C}(s^{-Q_+})\cap\mathbb{C}[[s^{-Q_+}]]$. 
Note also that $c_n(\theta;s|q,t)$ 
has at most (possibly multiple) poles along the divisors 
$s_j/s_i=q^{k}\quad(1\le i<j\le n;\ k\in\mathbb{Z}).$ 

These functions 
$c_n(\theta;s|q,t)$ 
can be determined inductively on $n$ 
starting from $c_1(0;s_1|q,t)=1$. 
Let us parametrize $\widetilde{\theta}\in M_{n+1}$ by 
a pair $(\theta,\nu)\in  M_n\times \mathbb{N}^n$ as follows:  
\begin{align}
\widetilde{\theta}_{i,j}=\theta_{i,j}\quad(1\le i<j\le n),
\quad 
\widetilde{\theta}_{i,n+1}=\nu_i\quad(i=1,\ldots,n). 
\end{align}
Then, $c_{n+1}(\theta;s|q,t)$ is expressed as 
\begin{align}\label{eq:recfmcn}
c_{n+1}(\widetilde{\theta};s|q,t)
=
\dprod{1\le i<j\le n+1}{}\,
\dfrac{(ts_j/s_i;q)_{\nu_i}}{(qs_j/s_i;q)_{\nu_i}}
\dprod{1\le i\le j\le n}{}\,
\dfrac{(q^{-\nu_j}qs_j/ts_i;q)_{\nu_i}}{(q^{-\nu_j}s_j/s_i;q)_{\nu_i}}
c_{n}(\theta;q^{-\nu}s|q,t).  
\end{align}

\begin{example}
The rational functions $c_n(\theta;s|q,t)$ for $n=2$ are 
given explicitly by 
\begin{align}
c_2(\theta_{12};s_1,s_2|q,t)
=
\dfrac{(ts_2/s_1;q)_{\theta_{12}}}
{(qs_2/s_1;q)_{\theta_{12}}}
\dfrac{(q^{-\theta_{12}}q/t)_{\theta_{12}}}
{(q^{-\theta_{12}})_{\theta_{12}}}
=
\dfrac{(ts_2/s_1;q)_{\theta_{12}}}
{(qs_2/s_1;q)_{\theta_{12}}}
\dfrac{(t;q)_{\theta_{12}}}{(q;q)_{\theta_{12}}}\,(q/t)^{\theta_{12}}.  
\end{align}
For $n=3$, 
\begin{align}
c_3(\theta_{12},\theta_{13},\theta_{23};s_1,s_2, s_3|q,t)
&=
\dfrac{(q^{\theta_{13}-\theta_{23}}ts_2/s_1;q)_{\theta_{12}}}
{(q^{\theta_{13}-\theta_{23}}qs_2/s_1;q)_{\theta_{12}}}
\dfrac{(q^{-\theta_{12}}q/t;q)_{\theta_{12}}}
{(q^{-\theta_{12}}q;q)_{\theta_{12}}}
\nonumber\\
&\quad\cdot
\dfrac{(ts_2/s_1;q)_{\theta_{13}}}
{(qs_2/s_1;q)_{\theta_{13}}}
\dfrac{(ts_3/s_1;q)_{\theta_{13}}}
{(qs_3/s_1;q)_{\theta_{13}}}
\dfrac{(ts_3/s_2;q)_{\theta_{23}}}
{(qs_3/s_2;q)_{\theta_{23}}}
\nonumber\\
&\quad\cdot
\dfrac{(q^{-\theta_{13}}q/t;q)_{\theta_{13}}}
{(q^{-\theta_{13}}q;q)_{\theta_{13}}}
\dfrac{(q^{-\theta_{23}}qs_2/ts_1;q)_{\theta_{13}}}
{(q^{-\theta_{23}}qs_2/s_1;q)_{\theta_{13}}}
\dfrac{(q^{-\theta_{23}}q/t;q)_{\theta_{23}}}
{(q^{-\theta_{23}}q;q)_{\theta_{23}}}. 
\end{align}
\end{example}

\par\medskip
For each $n=1,2,\ldots$, we introduce a formal power series 
\begin{align}\label{eq:defpn}
p_n(x;s|q,t)=
\dsum{\theta\in M_n}{}c_n(\theta; s|q,t)\ 
\dprod{1\le i<j\le n}{}(x_j/x_i)^{\theta_{ij}}
\in \mathbb{C}(s^{-Q_+})[[x^{-Q_+}]]
\end{align}
with leading coefficient 1.  
For each $\theta\in M_n$, define an element $\mu(\theta)\in Q_+$ by 
the formula
\begin{align}
\dprod{1\le i<j\le n}{} (x_j/x_i)^{\theta_{ij}}=x^{-\mu(\theta)}.  
\quad
\mu(\theta)=\dsum{1\le i<j\le n}{}\theta_{i,j}(\epsilon_i-\epsilon_j)\in Q_+.  
\end{align}
For each $\mu\in Q_+$, we denote by 
$M_n(\mu)$ the subset of $M_n$ consisting of all 
$\theta\in M_n$ such that $\mu(\theta)=\mu$; 
this set $M_n(\mu)$ is a finite set for each $\mu\in Q_+$.  
Then the power series expansion of $p_n(x;s|q,t)$ 
is given by 
\begin{align}\label{eq:exppn}
p_n(x;s|q,t)=\dsum{\mu\in Q_+}{}\,x^{-\mu} p_\mu(s), 
\quad
p_\mu(s)=
\dsum{\theta\in M_n(\mu)}{}\,c_n(\theta;s|q,t)
\quad(\mu\in Q_+). 
\end{align}
We also remark that the recurrence relation \eqref{eq:recfmcn} 
gives rise to a recurrence formula for $p_n(x;s|q,t)$: 
\begin{align}\label{eq:recpn}
p_{n+1}(x;s|q,t)
&=\dsum{\nu\in\mathbb{N}^n}{}
\dprod{i=1}{n}
\dfrac{(t;q)_{\nu_i}}{(q;q)_{\nu_i}}
\dprod{1\le i<j\le n+1}{}\,
\dfrac{(ts_j/s_i;q)_{\nu_i}}{(qs_j/s_i;q)_{\nu_i}}
\dprod{1\le i<j\le n}{}\,
\dfrac{(q^{-\nu_j+1}s_j/ts_i;q)_{\nu_i}}{(q^{-\nu_j}s_j/s_i;q)_{\nu_i}}
\nonumber\\
&\qquad\quad\cdot
\dprod{i=1}{n}(qx_{n+1}/tx_i)^{\nu_i}\ 
p_n(x;q^{-\nu}s|q,t), 
\end{align}
where we simply wrote $p_{n+1}(x;s|q,t)$ in place of 
$p_{n+1}(x_1,\ldots,x_{n+1};s_1,\ldots,s_{n+1}|q,t)$ 
by using the same symbols $x$, $s$ for $(n+1)$ variables.  

We introduce another formal power series 
$\varphi_n(x;s|q,t)\in \mathbb{C}[[s^{-Q_+}]][[x^{-Q_+}]]$ 
with a different 
normalization: 
\begin{align}\label{eq:defphin}
\varphi_n(x;s|q,t)
&=
\dprod{1\le i<j\le n}{}
\dfrac{(qs_j/s_i;q)_\infty}{(qs_j/ts_i;q)_\infty}\ 
p_n(x;s|q,t)
\nonumber\\
&=
\dprod{1\le i<j\le n}{}
\dfrac{(qs_j/s_i;q)_\infty}{(qs_j/ts_i;q)_\infty}\ 
\dsum{\theta\in M_n}{}c_n(\theta; s|q,t)\ 
\dprod{1\le i<j\le n}{}(x_j/x_i)^{\theta_{ij}}. 
\end{align}
Note that $\varphi_1(x;s|q,t)=1$.  For $n=2$, 
$\varphi_2(x;s|q,t)$ is expressed in terms of 
$q$-hypergeometric series:
\begin{align}
\varphi_2(x_1,x_2;s_1,s_2|q,t)
&=\dfrac{(qs_2/s_1;q)_\infty}{(qs_2/ts_1;q)_\infty}\ 
\dsum{k=0}{\infty}
\dfrac{(ts_2/s_1;q)_{k}}
{(qs_2/s_1;q)_{k}}
\dfrac{(t;q)_{k}}{(q;q)_{k}}\dfrac{}{}(qx_2/tx_1)^k
\nonumber\\
&=\dfrac{(qs_2/s_1;q)_\infty}{(qs_2/ts_1;q)_\infty}\ 
{}_{2}\phi_{1}\!
\left[\begin{matrix}
t,\ ts_2/s_1 \\
qs_2/s_1
\end{matrix};\,
q, qx_2/tx_1
\right]. 
\end{align}
By the definition \eqref{eq:defphin}, 
the leading coefficient of $\varphi_n(x;s|q,t)$ in $x$ variables is 
\begin{align}
\varphi_n(x;s|q,t)\big|_{x_{i+1}/x_i=0\ (i=1,\ldots,n-1)}
=
\dprod{1\le i<j\le n}{}
\dfrac{(qs_j/s_i;q)_\infty}{(qs_j/ts_i;q)_\infty}. 
\end{align} 
On the other hand, 
the leading coefficient as a formal power series in $s$ variables 
is determined by using the recurrence formula \eqref{eq:recpn} 
and the $q$-binomial theorem: 
\begin{align}
\varphi_n(x;s|q,t)\big|_{s_{i+1}/s_i=0\ (i=1,\ldots,n-1)}
=
\dprod{1\le i<j\le n}{}
\dfrac{(qx_j/x_i;q)_\infty}{(qx_j/tx_i;q)_\infty}.  
\end{align}

\subsection{Relation to Macdonald polynomials}
The definition of the power series $p_n(x;s|q,t)$ 
and $\varphi_n(x;s|q,t)$ originates from the 
tableau representation (or the restriction formula) 
of the Macdonald polynomials. 
Recall from \cite{M1995} that the monic Macdonald polynomial 
$P_\lambda(x|q,t)$ 
associated with a partition $\lambda=(\lambda_1,\ldots,\lambda_n)$ 
is expressed explicitly as
\begin{align}\label{eq:tabrep}
P_{\lambda}(x|q,t)=
\dsum{\phi=\mu^{(0)}\subseteq\mu^{(1)}\subseteq\cdots
\subseteq\mu^{(n)}=\lambda}{}
\dprod{k=1}{n}\ \psi_{\mu^{(k)}/\mu^{(k-1)}}(q,t)\, \ x_k^{|\mu^{(k)}/\mu^{(k-1)}|}
\end{align}
summed over increasing sequences of partitions of $n$ steps form 
$\phi$ to $\lambda$, where 
$\psi_{\lambda/\mu}(q,t)$ is defined for each skew partition 
$\lambda/\mu$ ($\mu\subseteq\lambda$) by 
\begin{align}
\psi_{\lambda/\mu}(q,t)
=
\dprod{1\le i<j\le n}{}
\dfrac{(q^{\mu_i-\lambda_j+1}t^{j-i-1};q)_{\lambda_i-\mu_i}}
{(q^{\mu_i-\lambda_j}t^{j-i};q)_{\lambda_i-\mu_i}}
\dprod{1\le i\le j\le n-1}{}
\dfrac{(q^{\mu_i-\mu_j}t^{j-i+1};q)_{\lambda_i-\mu_i}}
{(q^{\mu_i-\mu_j+1}t^{j-i};q)_{\lambda_i-\mu_i}}. 
\end{align}
Note that $\psi_{\lambda/\mu}(q,t)=0$ unless 
$\lambda/\mu$ is a horizontal strip, namely, 
$\mu_i\ge \lambda_{i+1}$ for all $i$. 

By interpreting the representation 
\eqref{eq:tabrep} in terms of column strict tableaux, 
we denote by  $\theta_{i,j}=\mu^{(j)}_i-\mu^{(j-1)}_i$ the number 
of $j$'s in the $i$-th row for $1\le i<j\le n$.  
Then the partitions $\mu^{(k)}$ ($k=1,\ldots,n$) are parametrized 
by $\theta=(\theta_{i,j})_{i,j=1}^{n}\in M_n$ as 
$\mu^{(k)}_i=\lambda_i-\sum_{a>k}\theta_{i,a}$ 
for $1\le i<k$.  
Through this parametrization together with $s_i=t^{n-i}q^{\lambda_i}$
($i=1,\ldots,n$), one can directly verify that 
\begin{align}
\dprod{k=1}{n}\ \psi_{\mu^{(k)}/\mu^{(k-1)}}(q,t)=c_n(\theta;s|q,t),
\quad
\dprod{k=1}{n}x_k^{|\mu^{(k)}/\mu^{(k-1)}|}=
x^\lambda\dprod{1\le i<j\le n}{}(x_j/x_i)^{\theta_{ij}}. 
\end{align}

We now take $c_n(\theta;s|q,t)$ for an arbitrary $\theta\in M_n$,
and consider the specialization $s=t^\delta q^{\lambda}$ by 
a partition $\lambda$, assuming that $t$ is generic. 
Note that the factor
$(q^{\sum_{a>k}(\theta_{i,a}-\theta_{j,a})}ts_j/s_i;q)_{\theta_{i,k}}$ 
in the definition \eqref{eq:defcn} with $j=i+1$ reduces to 
$(q^{\sum_{a>k}(\theta_{i,a}-\theta_{i+1,a})-\lambda_i+\lambda_{i+1}};q)_{\theta_{i,k}}$ 
under this specialization.  
Suppose that
\begin{align}
\dprod{1\le i<k\le n}{}
(q^{\sum_{a>k}(\theta_{i,a}-\theta_{i+1,a})-\lambda_i+\lambda_{i+1}};q)_{\theta_{i,k}}
\ne 0.  
\end{align}
The factors with $k=n$ then imply 
$\dprod{1\le i<n}{}
(q^{-\lambda_i+\lambda_{i+1}};q)_{\theta_{i,n}}
\ne 0$, 
and hence 
$0\le \theta_{i,n}\le \lambda_i-\lambda_{i+1}$ ($1\le i<n$) since 
$\lambda_i\ge\lambda_{i+1}$. 
Starting from $k=n$, by descending induction on $k$, we see 
\begin{align}
0\le \theta_{i,k}\le 
\lambda_{i}-\lambda_{i+1}-\dsum{a>k}{}(\theta_{i,a}-\theta_{i+1,a})
=\mu^{(k)}_i-\mu^{(k)}_{i+1}
\quad(1\le i<k), 
\end{align}
which confines $\theta$ in a finite subset of $M_n$ so that 
the monomial 
$x^{\lambda}\prod_{1\le i\le n}(x_j/x_i)^{\theta_{i,j}}$ remains 
as a polynomial in $x$.   This shows that 
$x^{\lambda}\,p_n(x;t^\delta q^\lambda)=P_{\lambda}(x|q,t)$ 
for any partition $\lambda=(\lambda_1,\ldots,\lambda_n)$, 
The normalization factor for $\varphi_n(x;s|q,t)$ takes care 
of the reciprocal of the evaluation formula
\begin{align}
P_\lambda(t^\delta|q,t)
&=t^{\sum_{i} (i-1)\lambda_i}
\dprod{1\le i<j\le n}{}
\dfrac{(t^{j-i+1};q)_{\lambda_i-\lambda_j}}{(t^{j-i};q)_{\lambda_i-\lambda_j}}
\nonumber\\
&
=t^{\br{\delta,\lambda}}
\dprod{i=1}{n}
\dfrac{(q/t;q)_{\infty}}{(q/t^i;q)_\infty}
\dprod{1\le i<j\le n}{}
\dfrac{(qs_j/ts_i;q)_{\infty}}{(qs_j/s_i;q)_{\infty}}
\end{align}
\begin{thm}\label{thm:MP}
Suppose that the parameter $t$ is generic in the sense 
that $t^k\notin q^{\mathbb{Z}}$ for $k=1,\ldots,n-1$.    
Then, for any partition $\lambda=(\lambda_1,\ldots,\lambda_n)$, 
the formal power series $x^\lambda p_n(x;s|q,t)$ and 
$x^\lambda\,\varphi_n(x;s|q,t)$ specialized 
by $s=t^\delta q^\lambda=(t^{n-1}q^{\lambda_1},\ldots,q^{\lambda_n})$ 
recover the Macdonald polynomial $P_\lambda(x|q,t)$ 
and 
a constant multiple of the normalized 
Macdonald polynomial $\widetilde{P}_\lambda(x|q,t)$, respectively\,$:$ 
\begin{align}
x^{\lambda}\,p_n(x;t^\delta q^{\lambda}|q,t)&=P_\lambda(x;|q,t), 
\nonumber\\
x^{\lambda}\,\varphi_n(x;t^\delta q^{\lambda}|q,t)&=
t^{\br{\delta,\lambda}}\dprod{i=1}{n}\dfrac{(q/t;q)_\infty}{(q/t^i;q)_\infty}
\, 
\widetilde{P}_{\lambda}(x|q,t).  
\end{align}
\end{thm}
\qed

\subsection{Eigenfunction equation in $x$ variables}
With the complex variables $\lambda$ such that $s=t^\delta q^\lambda$, 
we consider the eigenfunction equation 
\begin{align}
D^x(u)\,f(x;s)=f(x;s)\, \dprod{i=1}{n}(1-us_i),\quad 
f(x;s)=x^{\lambda}\varphi(x;s)\in x^\lambda\mathbb{C}(s^{-Q_+})[[x^{-Q_+}]], 
\end{align}
or equivalently,
\begin{align}
E^x(u)\,\varphi(x;s)=\varphi(x;s)\, \dprod{i=1}{n}(1-us_i),\quad 
\varphi(x;s)\in\mathbb{C}(s^{-Q_+})[[x^{-Q_+}]]
\end{align}
with leading coefficient $\varphi_0(s)=1$.  
This equation is equivalent to an infinite set of linear recurrence 
relations for the coefficients $\varphi_\mu(s)\in \mathbb{C}(s^{-Q_+})$ 
of $\varphi(x;s)$.  When we take 
$\varphi(x;s)=p_n(x;s|q,t)$, we already know 
these relations hold under the specialization 
$s=t^\delta q^\lambda$ for any partition $\lambda$ of length $\le n$,  
since $p_n(x;t^\delta q^\lambda|q,t)$ coincides with $P_{\lambda}(x|q,t)$. 
This means that each of the recurrence 
relations holds for arbitrary partitions $\lambda$, hence 
it holds as an 
identity of rational functions.  Hence we have 

\begin{thm}\label{thm:eigenDfn}
For the complex parameters $\lambda$ with $s=t^\delta q^\lambda$, 
the formal power series $x^\lambda p_n(x;s|q,t)$ satisfies
the eigenfunction equation
\begin{align}
D^x(u)\,x^\lambda p_n(x;s|q,t)=
x^\lambda p_n(x;s|q,t)\dprod{i=1}{n}(1-us_i). 
\end{align}
\end{thm}

Since $x^\lambda p_{n}(x;s|q,t)$ is a formal solution 
with leading coefficient 1, 
by Theorem \ref{thm:FR} we know that, 
for each $\mu\in Q_+$, the rational function 
\begin{align}
p_\mu(s)
=\dsum{\theta\in M_n(\mu)}{}\, c_{n}(\theta; s|q,t)
\end{align}
has at most simple poles along the divisors 
$s_j/s_i=q^{-k-1}\quad(i<j; k=0,1,\ldots)$. 
This means that the poles at $s_j/s_i=q^{k}$ with $k=0,1,2,\ldots$, which 
are apparent in $c_n(\theta;s|q;t)$, cancel out after the summation 
over all $\theta\in M_n(\mu)$.

\subsection{Symmetry with respect to $t\leftrightarrow q/t$}
Let us define a formal power series $\psi(x;s)$ by 
\begin{align}
p_n(x;s|q,t)
=\dprod{1\le i<j\le n}{}\dfrac{(qx_j/x_i;q)_\infty}{(qx_j/tx_i;q)_\infty}\,\psi(x;s),
\quad \psi_0(s)=1.  
\end{align}
Then by Theorem \ref{thm:FRpsi}, (2), 
\begin{align}
\psi(x;s)
=
\dprod{1\le i<j\le n}{}
\dfrac{(qx_j/tx_i;q)_\infty}{(qx_j/x_i;q)_\infty}\,
\dsum{\theta\in M_n}{}c_n(\theta;s|q,t)
\dprod{1\le i<j\le n}{}(x_j/x_i)^{\theta_{ij}}
\end{align}
is invariant under the change of parameters $t\leftrightarrow q/t$.  
\begin{prop}\label{prop:t-q/t}
The formal solution $p_n(x;s|q,t)$ of \eqref{eq:eigenEphi} 
with leading coefficient 1 satisfies the symmetry relation
\begin{align}\label{eq:transfp}
p_n(x;s|q,t)=
\dprod{1\le i<j\le n}{}
\dfrac{(tx_j/x_i;q)_\infty}{(qx_j/tx_i;q)_\infty}\,p_n(x;s|q,q/t)
\end{align}
with respect to the change of parameters $t\leftrightarrow q/t$.  
Namely, we have the transformation formula
\begin{align}\label{eq:transftqt}
&
\dsum{\theta\in M_n}{}c_n(\theta;s|q,t)
\dprod{1\le i<j\le n}{}(x_j/x_i)^{\theta_{ij}}
\nonumber\\
&
=
\dprod{1\le i<j\le n}{}
\dfrac{(tx_j/x_i;q)_\infty}{(qx_j/tx_i;q)_\infty}\,
\dsum{\theta\in M_n}{}c_n(\theta;s|q,q/t)
\dprod{1\le i<j\le n}{}(x_j/x_i)^{\theta_{ij}}. 
\end{align}
\end{prop}
\par\noindent
This formula for $n=2$ 
\begin{align}
{}_{2}\phi_{1}\!
\left[\begin{matrix}
t,\ ts_2/s_1 \\
qs_2/s_1
\end{matrix};\,
q, qx_2/tx_1
\right]
=
\dfrac{(tx_2/x_1)_\infty}{(qx_2/tx_1)_\infty}\ 
{}_{2}\phi_{1}\!
\left[\begin{matrix}
q/t,\ qs_2/ts_1 \\
qs_2/s_1
\end{matrix};\,
q, tx_2/x_1
\right]
\end{align}
is the $q$-Euler transformation formula for ${}_2\phi_1$. 
In terms of the formal solution $\varphi_n(x;s|q,t)$, 
the transformation formula mentioned above is expressed as 
\begin{align}\label{eq:transfphi}
\varphi_n(x;s|q,t)=
\dprod{1\le i<j\le n}{}
\dfrac{(tx_j/x_i;q)_\infty}{(qx_j/tx_i;q)_\infty}\,
\dfrac{(ts_j/s_i;q)_\infty}{(qs_j/ts_i;q)_\infty}\ 
\varphi_n(x;s|q,q/t). 
\end{align}

\begin{rem}\rm 
In this section, we made use of the explicit formula 
for Macdonald polynomials to prove that $p_n(x;s|q,t)$ 
and $\varphi_n(x;s|q,t)$ are formal solutions of the 
eigenfunction equation in the $x$ variables. 
As we will see in the next section on, however, 
Theorem \ref{thm:eigenDfn} 
as well as Proposition \ref{prop:t-q/t} 
can be proved without relying on the theory of Macdonald 
polynomials. 
\end{rem}


\section{Recurrence by Jackson integrals}

In this section, we show that the explicit formal solution 
$f(x;s)=x^\lambda\varphi_n(x;s|q,t)$ 
of the eigenfunction equation
in $x$ variables
essentially solves the bispectral problem
\begin{align}\label{eq:bispf}
D^x(u)\,f(x;s)=f(x;s)\,\dprod{i=1}{n}(1-us_i),
\quad
D^s(u)\,f(x;s)=f(x;s)\,\dprod{i=1}{n}(1-ux_i).  
\end{align}

\subsection{Preliminary remarks}
Let $e_n(x;s)$ be 
a (possibly multi-valued) meromorphic function in $(x,s)$ 
such that
\begin{align}\label{eq:qDEe}
T_{q,x_i}e_n(x;s)=e_n(x;s)\,s_{i}t^{-n+i},\quad 
T_{q,s_i}e_n(x;s)=e_n(x;s)\,x_{i}t^{-n+i}\quad(i=1,\ldots,n)
\end{align}  
and that $e_n(x;s)=e_n(s;x)$. 
In order to fix the idea, we choose now the function 
\begin{align}
e_n(x;s)=q^{\sum_{i=1}^{n}\kappa_i\lambda_i},
\end{align}
defined in terms of 
additive complex variables 
$\kappa=(\kappa_1,\ldots,\kappa_n)$
and $\lambda=(\lambda,\ldots,\lambda_n)$ 
such that 
$x_i=t^{n-i}q^{\kappa_i}$, $s_i=t^{n-i}q^{\lambda_i}$ ($i=1,\ldots,n$).  
From the relation $x=t^\delta q^\kappa$ and $s=t^\delta q^\lambda$, 
we have
\begin{align}
e_n(x;s)=x^{\lambda}t^{-\br{\delta,\lambda}}=s^{\kappa}t^{-\br{\delta,\kappa}}.  
\end{align}
We remark that the function 
\begin{align}
e_n(x;s)=\dprod{i=1}{n}
\dfrac{\theta(x_it^{n-i};q)\,\theta(s_it^{n-i};q)}
{\theta(x_is_i;q)}. 
\end{align}
defined by theta functions is an alternative choice for $e_n(x;s)$.  

We regard $\varphi_n(x;s|q,t)$ as a formal power series 
in $(x_2/x_1,\ldots,x_n/x_{n-1}; s_2/s_1,\ldots,s_n/s_{n-1})$: 
\begin{align}
\varphi_n(x;s|q,t)=\dsum{\mu,\nu\in Q_+}{} x^{-\mu}s^{-\nu} \varphi_{\mu,\nu}
\in \mathbb{C}[[x^{-Q_+}]][[s^{-Q_+}]].  
\end{align}
Our goal is to show that the formal power series
\begin{align}
f_n(x;s|q,t)=e_n(x;s)\,\varphi_n(x;s|q,t)\in 
e_n(x;s)\mathbb{C}[[x^{-Q_+}]][[s^{-Q_+}]]
\end{align}
is a solution of the bispectral problem \eqref{eq:bispf}. 
Since $e_n(x;s)=x^{\lambda}t^{-\br{\delta,\lambda}}$, 
we already know that $f_n(x;s|q,t)$ is a solution of the 
eigenfunction equation in $x$ variables.  We 
need to show that this formal power series simultaneously solves 
the eigenfunction equation in $s$ variables. 
For that purpose we will make use of the inductive construction 
of eigenfunctions by Jackson integrals.  

\par\medskip
Before investigating the general case, we observe the  
case $n=2$ in some detail.  Recall that 
\begin{align}
\varphi_2(x;s|q,t)
&=\dfrac{(qs_2/s_1;q)_\infty}{(qs_2/ts_1;q)_\infty}\ 
{}_{2}\phi_{1}\!
\left[\begin{matrix}
t,\ ts_2/s_1 \\
qs_2/s_1
\end{matrix};\,
q, qx_2/tx_1
\right]
\nonumber\\
&=\dfrac{(qs_2/s_1;q)_\infty}{(qs_2/ts_1;q)_\infty}\ 
\dsum{k=0}{\infty}
\dfrac{(t;q)_{k}(ts_2/s_1;q)_{k}}{(qs_2/s_1;q)_{k}(q;q)_{k}}
(qx_2/tx_1)^k.  
\end{align}
This series defines a holomorphic function on the domain 
$|qx_2/tx_1|<1$ and 
$qs_2/ts_1\notin q^{-\mathbb{N}}.$ 
Suppose that $|t|<1$.  Then by 
\begin{align}
\dfrac{(t;q)_k}{(qs_2/s_1;q)_k}
=
\dfrac{(t;q)_\infty (q^{k+1}s_2/s_1;q)_\infty}
{(qs_2/s_1;q)_\infty (q^kt;q)_\infty}
=
\dfrac{(t;q)_\infty}{(qs_2/s_1;q)_\infty}
\dsum{l=0}{\infty}
\dfrac{(qs_2/ts_1;q)_l}{(q;q)_l} q^{kl}t^l, 
\end{align}
we have
\begin{align}
\varphi_2(x;s|q,t)
&=
\dfrac{(t;q)_\infty}{(qs_2/ts_1;q)_\infty}\ 
\dsum{k,l}{}
\dfrac{(qs_2/ts_1;q)_l}{(q;q)_l} 
\dfrac{(ts_2/s_1;q)_{k}}{(q;q)_{k}}
(qx_2/tx_1)^k
q^{kl}
t^l
\nonumber\\
&=
\dfrac{(t;q)_\infty}{(qs_2/ts_1;q)_\infty}\ 
\dsum{l}{}
\dfrac{(qs_2/ts_1;q)_l 
}{(q;q)_l}
\dfrac{(q^{l+1}x_2s_2/x_1s_1;q)_\infty}{(q^{l+1}x_2/tx_1;q)_\infty}
t^l
\nonumber\\
&
=
\dfrac{(t;q)_\infty(qx_2s_2/x_1s_1;q)_\infty}
{(qs_2/ts_1;q)_\infty(qx_2/tx_1;q)_\infty}
\dsum{l}{}
\dfrac{(qs_2/ts_1;q)_l(qx_2/tx_1;q)_l}{(q;q)_l(qx_2s_2/x_1s_1;q)_l}
t^l. 
\end{align}
Hence
\begin{align}\label{eq:n=2final}
\varphi_2(x;s|q,t)=
\dfrac{(t;q)_\infty(qx_2s_2/x_1s_1;q)_\infty}
{(qx_2/tx_1;q)_\infty(qs_2/ts_1;q)_\infty}\,
{}_2\phi_{1}\!
\left[\begin{matrix}
qx_2/tx_1, qs_2/ts_1 \\
qx_2s_2/x_1s_1
\end{matrix};\,
q,t
\right]
\quad(|t|<1).  
\end{align}
This expression represents a meromorphic function 
in $(x_2/x_1,s_2/s_1)\in\mathbb{C}\times\mathbb{C}$ 
at most with simple poles at
$qx_2/tx_1\in q^{-\mathbb{N}}$ and 
$qs_2/ts_1\in q^{-\mathbb{N}}$. 
Furthermore $\varphi_2(x;s|q,t)$ is manifestly symmetric 
with respect to $x$ and $s$ variables. 
This immediately implies that the meromorphic function 
$f_2(x;s|q,t)=e_2(x;s)\varphi_2(x;s|q,t)$ actually 
solves the bispectral problem.  See \cite[Proposition 5.4]{S2010} also.

For general $n$, it is a challenging problem to find an 
analytic expression that directly implies symmetry between 
$x$ and $s$ variables;  the case  $n=3$ will be studied explicitly in Section 7.
In the following, we will show that $f_n(x;s|q,t)$ 
satisfies the eigenfunction equation for $s$ variables 
by means of the Jackson integral representation.  
The symmetry with respect to $x$ and $s$ 
variables will be derived as a consequence of the fact 
that it satisfies the two eigenfunction equations 
simultaneously.  

\subsection{Recurrence by Jackson integrals} 

In order to prove that 
$f_n(x;s|q,t)=e_n(x;s)$ $\times\varphi_n(x;s|q,t)$ satisfies the 
eigenfunction equation in $s$ variables, 
by exchanging the roles of $x$ and $s$ variables  
we show that 
\begin{align}
f_n(s;x|q,t)=e_n(s;x)\varphi_n(s;x|q,t)
\end{align} 
satisfies the eigenfunction equation in the $x$ variables.  
For the proof we make use of a result of \cite{MN1997} 
concerning the integral representation of eigenfunctions 
for Ruijsenaars-Macdonald operators.  
We remark that the method of \cite{MN1997} is applicable 
to our formal setting as well, since it is based on 
the $q$-difference equation satisfied by a kernel function. 

\par\medskip
Let $g(y)$ be a joint eigenfunction of 
Ruijsenaars-Macdonald operators in $n$ variables $y=(y_1,\ldots,y_n)$ 
such that 
\begin{align}
D^y(u)g(y)=g(y)\,\dprod{i=1}{n}(1-ut^{n-i}q^{\lambda_i})
\end{align}
with complex parameters $\lambda=(\lambda_1,\ldots,\lambda_n)$.  
Then, for $m\ge n$, 
we define a function $f(x)$ in $m$
variables $x=(x_1,\ldots,x_{m})$ by an integral 
transformation of the form 
\begin{align}\label{eq:ITf}
f(x)
=
\dint{}{}
H_{\kappa}(x;y) \,g(y)\,d\omega(y)
\end{align}
with a complex parameter $\kappa$,   
where $d\omega(y)$ is a $q$-invariant measure on $\mathbb{T}^n_y$.    
We choose the kernel function so that 
\begin{align}
H_{\kappa}(x;y)\equiv 
(x_1\cdots x_{m})^\kappa\ 
\dprod{i=1}{n}
\dprod{j=1}{m}
\dfrac{(tx_j/y_i;q)_\infty}{(x_j/y_i;q)_\infty}
\dprod{1\le i,j\le n;\,i\ne j}{} 
\dfrac{(y_i/y_j;q)_\infty}{(ty_i/y_j;q)_\infty}\,
(y_1\cdots y_n)^{-\kappa}, 
\end{align}
where $\equiv$ means that the both sides coincides up to multiplication 
by a {\em quasi-constant}, namely, by a $q$-periodic function 
in all variables $x_j$ and $y_i$.
Note that $H_{\kappa}(x;y)$ is essentially the Cauchy kernel 
(with $y$ variables reversed)
multiplied by the weight function of orthogonality 
for Macdonald polynomials.  
We assume that the integral \eqref{eq:ITf} makes sense 
and that the actions of $T_{q,x_j}$ 
($j=1,\ldots,m$) commute with the integral.  
Then, it is known by \cite{MN1997} that $f(x)$ is as well  
a joint eigenfunction in $x$ variables and satisfies 
\begin{align}
D^x(u)f(x)=f(x)\dprod{j=1}{m}(1-ut^{m-j}q^{\lambda_j}), 
\end{align}
with additional 
parameters $\lambda_j=\kappa$ ($j=n+1,\ldots,m$). 

\par\medskip
We apply this integral transformation to formal power series, 
in the case where $m=n+1$ and $\kappa=\lambda_{n+1}$. 
For that purpose, we rewrite \eqref{eq:ITf} 
in terms of $\psi(y)$ and $\varphi(x)$ defined by 
\begin{align}
g(y)=y_1^{\lambda_1}\cdots y_n^{\lambda_n}\psi(y),
\quad
f(x)=x_1^{\lambda_1}\cdots x_{n+1}^{\lambda_{n+1}}\varphi(x),
\end{align}
respectively.  Then $\varphi(x)$ should be obtained from $\psi(y)$ 
by integral transformation 
\begin{align}\label{eq:ITphi}
\varphi(x)=\dint{}{}K_\lambda(x;y)\,\psi(y)\,d\omega(y) 
\end{align}
with a kernel such that 
\begin{align}
K_\lambda(x;y)&\equiv x_1^{-\lambda_1}\cdots x_n^{-\lambda_{n+1}}
H_{\lambda_{n+1}}(x;y) y^{\lambda_1}\cdots y_n^{\lambda_n}
\nonumber\\
&=
\dprod{i=1}{n}
\dprod{j=1}{n+1}
\dfrac{(tx_j/y_i;q)_\infty}{(x_j/y_i;q)_\infty}
\dprod{1\le i,j\le n;\,i\ne j}{} 
\dfrac{(y_i/y_j;q)_\infty}{(ty_i/y_j;q)_\infty}\,
\dprod{i=1}{n}
(y_i/x_i)^{\lambda_i-\lambda_{n+1}}. 
\end{align}
Multiplying the right hand side by the quasi-constant
\begin{align}
\dprod{1\le i<j\le n}{}
\dfrac{\theta(x_i/y_j;q)}{\theta(tx_i/y_j;q)}
\dfrac{\theta(ty_i/y_j;q)}{\theta(y_i/y_j;q)}
\dprod{i=1}{n}(y_i/x_i)^{(n-i)\beta}(q/t)^{\lambda_i-\lambda_{n+1}+(n-i)\beta}, 
\quad t=q^\beta, 
\end{align}
we choose the function 
\begin{align}\label{eq:defK}
K_\lambda(x;y)&=
\dprod{i=1}{n}
\dfrac{(tx_i/y_i;q)_\infty}{(x_i/y_i;q)_\infty}
\dprod{1\le i<j\le n+1}{}
\dfrac{(tx_j/y_i;q)_\infty}{(x_j/y_i;q)_\infty}
\dprod{1\le i<j\le n}{} 
\dfrac{(qy_j/x_i;q)_\infty}{(qy_j/tx_i;q)_\infty}
\nonumber\\
&\quad\cdot
\dprod{1\le i<j\le n}{} 
\dfrac{(y_j/y_i;q)_\infty}{(ty_j/y_i;q)_\infty}
\dfrac{(qy_j/ty_i;q)_\infty}{(qy_j/y_i;q)_\infty}\,
\dprod{i=1}{n}
(qy_i/tx_i)^{\lambda_i-\lambda_{n+1}+(n-i)\beta}
\end{align}
with parameters $\lambda=(\lambda_1,\ldots,\lambda_{n+1})$
for our kernel; a power of $q/t$ is introduced for 
the sake of normalization.  
Noting that this function has zeros at 
$y_i=q^{k}tx_i$ ($i=1,\ldots,n;\ k=0,1,2,\ldots$), as the $q$-invariant measure 
we make use of the Jackson integral
\begin{align}\label{eq:psitophi}
\varphi(x)=
\dfrac{1}{(1-q)^n}
\dint{tx_1}{\infty}
\!\cdots\!
\dint{tx_n}{\infty}
K_\lambda(x;y)\,\psi(y)\ 
\dfrac{d_qy_1}{y_1}
\!\cdots\!
\dfrac{d_qy_n}{y_n}. 
\end{align}
Then, by the argument of \cite{MN1997} applied to formal power series, 
we can formulate the following inductive construction of formal eigenfunctions. 
\begin{lem} \label{lem:indJack}
Let 
\begin{align}
\psi(x;s)\in\mathbb{C}[[x^{-Q_+}]][[s^{-Q_+}]],\quad
(x;s)=(x_1,\ldots,x_{n};s_1,\ldots,s_{n})
\end{align}
be a formal power series such that 
$g(x;s)=e_n(x;s)\psi(x;s)$ 
is a formal solution of the eigenfunction equation \eqref{eq:eigenDf} in 
$n$ variables $x=(x_1,\ldots, x_n)$. 
Define a formal power series 
\begin{align}
\varphi(x;s)\in\mathbb{C}[[x^{-Q_+}]][[s^{-Q_+}]],
\quad (x;s)=(x_1,\ldots,x_{n+1};s_1,\ldots,s_{n+1}), 
\end{align}
by the Jackson integral
\begin{align}\label{eq:psitophi2}
\varphi(x;s)=
\dfrac{1}{(1-q)^n}
\dint{tx_1}{\infty}
\!\cdots\!
\dint{tx_n}{\infty}
K_\lambda(x;y)\,\psi(y;s)\ 
\dfrac{d_qy_1}{y_1}
\!\cdots\!
\dfrac{d_qy_n}{y_n}, 
\end{align}
where $K_\lambda(x;y)$ is the kernel defined by \eqref{eq:defK}
with complex parameters 
$\lambda=(\lambda_1,\ldots,\lambda_{n+1})$
such that $s_i=t^{n+1-i}q^{\lambda_i}$ $($$i=1,\ldots,n+1$$)$. 
Then $f(x;s)=e_{n+1}(x;s)\varphi(x;s)$ 
is a formal solution of the eigenfunction equation in 
$(n+1)$ variables 
$x=(x_1,\ldots,x_{n+1})$.  
\end{lem}
\qed

The Jackson integral \eqref{eq:psitophi2} 
is in fact an $n$-tuple sum 
of the values of the integrand at the points 
\begin{align}
(y_1,\ldots,y_n)=(q^{-\nu_1-1}tx_1,\ldots,q^{-\nu_n-1}tx_n),
\quad \nu=(\nu_1,\ldots,\nu_n)\in\mathbb{N}^n. 
\end{align}
It is expressed as 
\begin{align}
\varphi(x;s)
&=
\dsum{\nu\in \mathbb{N}^n}{}
K_\lambda(x;q^{-\nu}tx/q)\psi(q^{-\nu} tx/q;s)
\nonumber\\
&=
\left(
\dfrac{(q;q)_\infty}{(q/t;q)_\infty}
\right)^n
\dprod{1\le i<j\le n+1}{}
\dfrac{(qx_j/x_i;q)_{\infty}}{(qx_j/tx_i;q)_{\infty}}
\nonumber\\
&
\cdot
\dsum{\nu\in \mathbb{N}^n}{}
\dprod{i=1}{n}
\dfrac{(q/t;q)_{\nu_i}}{(q;q)_{\nu_i}}
\dprod{1\le i<j\le n+1}{}
\dfrac{(qx_j/tx_i;q)_{\nu_i}}{(qx_j/x_i;q)_{\nu_i}}
\dprod{1\le i<j\le n}{}
\dfrac{(q^{-\nu_j}tx_j/x_i;q)_{\nu_i}}{(q^{-\nu_j}x_j/x_i;q)_{\nu_i}}
\nonumber\\
&
\quad\cdot
\dprod{i=1}{n}(ts_{n+1}/s_i)^{\nu_i}
\dprod{1\le i<j\le n}{}
\dfrac{(q^{\nu_i-\nu_j+1}x_j/tx_i;q)_{\infty}}
{(q^{\nu_i-\nu_j+1}x_j/x_i;q)_{\infty}}\, 
\psi(q^{-\nu}x;s), 
\end{align}
where $\psi(q^{-\nu}tx/q;s)$ is replaced by 
$\psi(q^{-\nu}x;s)$ by using the homogeneity of $\psi(x;s)$.

\par\medskip
In view of Lemma \ref{lem:indJack}, 
starting from $\phi_1(s;x|q,t)$=1, 
we define a sequence of formal power series 
\begin{align}
\phi_n(s;x|q,t)\in \mathbb{C}[[x^{-Q_+}]][[s^{-Q_+}]],
\quad (x;s)=(x_1,\ldots,x_{n};s_1,\ldots,s_{n})
\quad(n=1,2,\ldots)
\end{align}
inductively by 
\begin{align}\label{eq:indformphi}
&\phi_{n+1}(s;x|q,t)
\nonumber\\
&
=
\dprod{1\le i<j\le n+1}{}
\dfrac{(qx_j/x_i;q)_{\infty}}{(qx_j/tx_i;q)_{\infty}}
\nonumber\\
&
\quad \cdot\dsum{\nu\in \mathbb{N}^n}{}
\dprod{i=1}{n}
\dfrac{(q/t;q)_{\nu_i}}{(q;q)_{\nu_i}}
\dprod{1\le i<j\le n+1}{}
\dfrac{(qx_j/tx_i;q)_{\nu_i}}{(qx_j/x_i;q)_{\nu_i}}
\dprod{1\le i<j\le n}{}
\dfrac{(q^{-\nu_j}tx_j/x_i;q)_{\nu_i}}{(q^{-\nu_j}x_j/x_i;q)_{\nu_i}}
\nonumber\\
&
\quad\cdot
\dprod{i=1}{n}(ts_{n+1}/s_i)^{\nu_i}
\dprod{1\le i<j\le n}{}
\dfrac{(q^{\nu_i-\nu_j+1}x_j/tx_i;q)_{\infty}}
{(q^{\nu_i-\nu_j+1}x_j/x_i;q)_{\infty}}\, 
\phi_n(s;q^{-\nu}x|q,t). 
\end{align}
Then we obtain a sequence of formal solutions
\begin{align}
e_n(x;s)\phi_n(s;x|q,t)\in e_n(x;s) \mathbb{C}[[x^{-Q_+}]][[s^{-Q_+}]]
\end{align}
of the eigenfunction equations \eqref{eq:eigenDf} in $x$ variables. 
By the $q$-binomial theorem, 
we can determine 
the leading coefficient of $\phi_n(s;x|q,t)$ inductively as
\begin{align}
\phi_n(s;x|q,t)\big|_{x_{i+1}/x_i=0\ (i=1,\ldots,n-1)}
=
\dprod{1\le i<j\le n}{}
\dfrac{(qs_j/s_i;q)_\infty}{(ts_j/s_i;q)_\infty}. 
\end{align}
In view of Theorem \ref{thm:FRpsi}, (2), we introduce
\begin{align}
\psi_n(s;x|q,t)=
\dprod{1\le i<j\le n}{}
\dfrac{(qx_j/tx_i;q)_\infty}{(qx_j/x_i;q)_\infty}
\dfrac{(ts_j/s_i;q)_\infty}{(qs_j/s_i;q)_\infty}\,
\phi_n(s;x|q,t)
\quad(n=1,2,\ldots)
\end{align}
so that $\psi_n(s;x|q,t)$ has leading coefficient 1 in $x$, and 
becomes invariant under the change of parameters $t\leftrightarrow q/t$. 
We rewrite the recurrence formula \eqref{eq:indformphi} for 
$\phi_n(s;x|q,t)$ into 
that for $\psi_n(s;x|q,t)$: 
\begin{align}\label{eq:indformpsi}
\psi_{n+1}(s;x|q,t)
&=
\dprod{i=1}{n}
\dfrac{(ts_{n+1}/s_i;q)_{\infty}}{(qs_{n+1}/s_i;q)_{\infty}}
\dsum{\nu\in \mathbb{N}^n}{}
\dprod{i=1}{n}
\dfrac{(q/t;q)_{\nu_i}}{(q;q)_{\nu_i}}
\dprod{1\le i<j\le n+1}{}
\dfrac{(qx_j/tx_i;q)_{\nu_i}}{(qx_j/x_i;q)_{\nu_i}}
\nonumber\\
&
\quad\cdot\dprod{1\le i<j\le n}{}
\dfrac{(q^{-\nu_j}tx_j/x_i;q)_{\nu_i}}{(q^{-\nu_j}x_j/x_i;q)_{\nu_i}}
\dprod{i=1}{n}(ts_{n+1}/s_i)^{\nu_i}
\psi_n(s;q^{-\nu}x|q,t). 
\end{align}
By the symmetry of $\psi_n(s;x|q,t)$, 
we can exchange $t$ and $q/t$ so that 
\begin{align}\label{eq:indformpsi2}
\psi_{n+1}(s;x|q,t)
&=
\dprod{i=1}{n}
\dfrac{(qs_{n+1}/ts_i;q)_{\infty}}{(qs_{n+1}/s_i;q)_{\infty}}
\dsum{\nu\in \mathbb{N}^n}{}
\dprod{i=1}{n}
\dfrac{(t;q)_{\nu_i}}{(q;q)_{\nu_i}}
\dprod{1\le i<j\le n+1}{}
\dfrac{(tx_j/x_i;q)_{\nu_i}}{(qx_j/x_i;q)_{\nu_i}}
\nonumber\\
&
\quad\cdot\dprod{1\le i<j\le n}{}
\dfrac{(q^{-\nu_j}qx_j/tx_i;q)_{\nu_i}}{(q^{-\nu_j}x_j/x_i;q)_{\nu_i}}
\dprod{i=1}{n}(qs_{n+1}/ts_i)^{\nu_i}
\psi_n(s;q^{-\nu}x|q,t). 
\end{align}
By comparing this with the recurrence formula \eqref{eq:recpn} 
for $p_n(x;s|q,t)$, we conclude that
\begin{align}
\psi_n(s;x|q,t)=
\dprod{1\le i<j\le n}{}
\dfrac{(qs_j/ts_i;q)_\infty}{(qs_j/s_i;q)_\infty}\,p_n(s;x|q,t)
\end{align}
with the role of $x$ and $s$ variables exchanged, and hence
\begin{align}
\phi_n(s;x|q,t)&=
\dprod{1\le i<j\le n}{}
\dfrac{(qx_j/x_i;q)_\infty}{(qx_j/tx_i;q)_\infty}
\dfrac{(qs_j/ts_i;q)_\infty}{(ts_j/s_i;q)_\infty}
p_n(s;x|q,t)
\nonumber\\
&=
\dprod{1\le i<j\le n}{}
\dfrac{(qs_j/ts_i;q)_\infty}{(ts_j/s_i;q)_\infty}
\varphi_n(s;x|q,t).   
\end{align}
This shows that 
$e_n(x;s)\varphi_n(s;x|q,t)$ is a formal solution of the 
eigenfunction equation \eqref{eq:eigenDf} in $x$ variables. 
It also turns out that 
\begin{align}\label{eq:psi2exp}
\psi_n(s;x|q,t)
&=
\dprod{1\le i<j\le n}{}
\dfrac{(qs_j/ts_i;q)_\infty}{(qs_j/s_i;q)_\infty}\,p_n(s;x|q,t)
\nonumber\\
&=
\dprod{1\le i<j\le n}{}
\dfrac{(qx_j/tx_i;q)_\infty}{(qx_j/x_i;q)_\infty}
\dfrac{(qs_j/ts_i;q)_\infty}{(qs_j/s_i;q)_\infty}
\varphi_n(s;x|q,t)
\end{align}
is invariant under the change of parameters $t\leftrightarrow q/t$,
which reproves the transformation formulas \eqref{eq:transfp} 
and \eqref{eq:transfphi}.

\begin{thm}\label{thm:JIT}
The sequence of formal power series 
$\varphi_n(s;x|q,t)$ $($n=1,2,\ldots$)$
satisfies 
the following recurrence relation 
by Jackson integral transformations\,$:$ 
\begin{align}
\varphi_{n+1}(s;x|q,t)
&=
\left(
\dfrac{(q/t;q)_\infty}{(1-q)(q;q)_\infty}
\right)^n
\dprod{i=1}{n}\,
\dfrac{(ts_{n+1}/s_i;q)_\infty}{(qs_{n+1}/ts_i;q)_\infty}\,
\nonumber\\
&
\quad\cdot
\dint{tx_1}{\infty}
\cdots 
\dint{tx_n}{\infty}
K_\lambda(x;y)\,\varphi_{n}(s;y|q,t)\,
\dfrac{d_qy_1}{y_1}
\cdots
\dfrac{d_qy_n}{y_n}
\end{align}
where $K_\lambda(x;y)$ is the kernel defined by \eqref{eq:defK}.  
\end{thm}
\qed

\noindent
This Jackson integral transformation corresponds 
to the following recurrence formula for 
$\varphi_n(s;x|q,t)$:
\begin{align}\label{eq:recformphi2}
\varphi_{n+1}(s;x|q,t)
&=
\dprod{i=1}{n}
\dfrac{(ts_{n+1}/s_i;q)_\infty}{(qs_{n+1}/s_i;q)_\infty}
\dprod{1\le i<j\le n+1}{}
\dfrac{(qx_j/x_i;q)_\infty}{(qx_j/tx_i;q)_\infty}
\nonumber\\
&\cdot
\dsum{\nu\in\mathbb{N}^n}{}\,
\dprod{i=1}{n}
\dfrac{(q/t;q)_{\nu_i}}{(q;q)_{\nu_i}}
\dprod{1\le i<j\le n+1}{}
\dfrac{(qx_j/tx_i;q)_{\nu_i}}
{(qx_j/x_i;q)_{\nu_i}}
\dfrac{(q^{-\nu_j}tx_j/x_i;q)_{\nu_i}}
{(q^{-\nu_j}x_j/x_i;q)_{\nu_i}}
\nonumber\\
&\cdot
\dprod{i=1}{n}
(ts_{n+1}/s_i)^{\nu_i}\,
\dprod{1\le i<j\le n}{}
\dfrac{(q^{\nu_i-\nu_j+1}x_j/tx_i;q)_{\nu_i}}
{(q^{\nu_i-\nu_j+1}x_j/x_i;q)_{\nu_i}}\, 
\varphi_n(s;q^{-\nu}x).  
\end{align}

\begin{thm}\label{thm:eigenDfnsx}
For each $n=1,2,\ldots$, 
the formal power series 
\begin{align}
f(x;s)&=x^\lambda\,\varphi_n(s;x|q,t)
\nonumber\\
&=x^\lambda 
\dprod{1\le i<j\le n}{}
\dfrac{(qx_j/x_i;q)_\infty}{(qx_j/tx_i;q)_\infty}\,p_n(s;x|q,t)
\in x^\lambda\mathbb{C}[[x^{-Q_+}]][[s^{-Q_+}]]
\end{align}
defined with parameters $\lambda$ such that $s=t^\delta q^\lambda$
is a formal solution of the eigenfunction equation 
\eqref{eq:eigenDf} in $x$ variables.  
\end{thm}
\qed

\subsection{Formal solution of the bispectral problem}

In Theorem \ref{thm:eigenDfn} and Theorem \ref{thm:eigenDfnsx}, 
we proved that 
$x^\lambda\varphi_n(x;s|q,t)$ 
and $x^\lambda\varphi_n(s;x|q,t)$ satisfy the eigenfunction 
equation in the $x$ variables, respectively.  
As we already remarked in Section 2, $\varphi_n(x;s|q,t)$ has 
the leading coefficient
\begin{align}
\varphi_n(x;s|q,t)\big|_{s_{i+1}/s_i=0 \ (i=1,\ldots,n-1)}=
\dprod{1\le i<j\le n}{}
\dfrac{(qx_j/x_i;q)_\infty}{(qx_j/tx_i;q)_\infty} 
\end{align}
in $s$ variables. 
This means that the two formal solutions 
$x^\lambda\varphi_n(x;s|q,t)$ 
and $x^\lambda\varphi_n(s;x|q,t)$ 
of the eigenfunction equation in $x$ variables
have the same leading coefficient
\begin{align}
\dprod{1\le i<j\le n}{}
\dfrac{(qs_j/s_i;q)_\infty}{(qs_j/ts_i;q)_\infty}.  
\end{align}
Hence by Theorem \ref{thm:FF} we see 
$\varphi_n(x,s|q,t)$ and $\varphi_n(s,x|q,t)$ coincides 
as formal power series in $\mathbb{C}[[x^{-Q_+}]][[s^{-Q_+}]]$.

Summarizing these arguments, we have
\begin{thm}\label{bispect-thm}
Let $e_n(x;s)$ be a solution of the $q$-difference equations 
\eqref{eq:qDEe} with symmetry $e_n(x;s)=e_n(s;x)$.  
Then 
\begin{align}
f_n(x;s|q,t)=e_n(x;s)\varphi_{n}(x;s|q,t)
\in e_n(x;s) \mathbb{C}[[x^{-Q_+}]][[s^{-Q_+}]]
\end{align}
is a formal solution of the bispectral problem \eqref{eq:bispf}.  
The formal power series $\varphi_n(x;s|q,t)$ 
have leading coefficients 
\begin{align}
&
\varphi_{n}(x;s|q,t)\big|_{x_{i+1}/x_i=0\ (i=1,\ldots,n-1)}
=
\dprod{1\le i<j\le n}{}
\dfrac{(qs_j/s_i;q)_\infty}{(qs_j/ts_i;q)_\infty},
\nonumber\\
&
\varphi_{n}(x;s|q,t)\big|_{s_{i+1}/s_i=0\ (i=1,\ldots,n-1)}
=
\dprod{1\le i<j\le n}{}
\dfrac{(qx_j/x_i;q)_\infty}{(qx_j/tx_i;q)_\infty}, 
\end{align}
each of which determines the formal solution uniquely.  
It has symmetry 
\begin{align}
\varphi_n(x;s|q,t)=\varphi_n(s;x|q,t) 
\end{align}
between $x$ and $s$ variables, and transforms as 
\begin{align}
\varphi_n(x;s|q,t)=
\dprod{1\le i<j\le n}{}
\dfrac{(tx_j/x_i;q)_\infty}{(qx_j/tx_i;q)_\infty}
\dfrac{(ts_j/s_i;q)_\infty}{(qs_j/ts_i;q)_\infty}\,\varphi_n(x;s|q,q/t).  
\end{align}
under the change of parameters $t\leftrightarrow q/t$.
\end{thm}
\qed 

From the symmetry $\varphi_n(x;s|q,t)=\varphi_n(s;x|q,t)$, 
it turns out that 
the formal power series $\psi_n(x;s|q,t)$ of \eqref{eq:psi2exp} 
has symmetry $\psi_n(x;s|q,t)=\psi_n(s;x|q,t)$ as well. 
Namely, the formal power series
\begin{align}
\psi_n(x;s|q,t)
&=
\dprod{1\le i<j\le n}{}
\dfrac{(qx_j/tx_i;q)_\infty}{(qx_j/x_i;q)_\infty}\,
p_n(x;s|q,t)
\nonumber\\
&=
\dprod{1\le i<j\le n}{}
\dfrac{(qx_j/tx_i;q)_\infty}{(qx_j/x_i;q)_\infty}
\dsum{\theta\in M_n}{}
c_n(\theta;s|q,t)\dprod{1\le i<j\le n}{}(x_j/x_i)^{\theta_{ij}}
\end{align}
satisfies 
\begin{align}
\psi_{n}(x;s|q,t)=\psi_{n}(s;x|q,t)\quad\mbox{and}\quad
\psi_{n}(x;s|q,t)=\psi_{n}(x;s|q,q/t).  
\end{align}
We will prove later that this $\psi_n(x;s|q,t)$ 
represents a meromorphic function on 
$\mathbb{C}^{n-1}\times\mathbb{C}^{n-1}$ with coordinates 
$(x_2/x_1,\ldots,x_n/x_{n-1};s_2/s_1,\ldots,s_n/s_{n-1})$ 
at most with simple poles along
\begin{align}
q^{k+1}x_j/x_i=1,\quad
q^{k+1}s_j/s_i=1\quad(1\le i<j\le n;\ k=0,1,2,\ldots). 
\end{align}  


\section{Recurrence by $q$-difference operators}

In the previous section, we described the inductive 
structure for $\varphi_n(x;s|q,t)$ $(n=1,2,\ldots)$
in terms of Jackson integrals.  
The inductive structure of  $\varphi_n(x;s|q,t)$ can be 
reformulated as 
recurrence by Ruijsenaars-Macdonald operators {\em of row type}.  
By using this fact, we give an alternative proof of duality 
and bispectrality of $f_n(x;s)=e_n(x;s)\varphi_n(x;s|q,t)$ 
which does not rely on the theory of Macdonald polynomials. 
We also present an explicit formula for $\varphi_n(x;s|q,t)$ 
which manifestly shows (a part of) symmetry between 
$x$ and $s$ variables. 
We first recall from \cite{NS2012} the definition and 
some basic results on the family of Ruijsenaars-Macdonald operators 
of row type. 

\subsection{Ruijsenaars-Macdonald operators of row type}

For each $l=0,1,2\ldots$, we introduce the following 
$q$-difference operators $H_l^{x}$ {\em of row type}: 
\begin{align}
H^x_l=\dsum{\nu\in\mathbb{N}^n;\, |\nu|=l}{}
\dprod{1\le i<j\le n}{}
\dfrac{q^{\nu_i}x_i-q^{\nu_j}x_j}{x_i-x_j}
\dprod{1\le i,j\le n}{}
\dfrac{(tx_i/x_j;q)_{\nu_i}}{(qx_i/x_j;q)_{\nu_i}}
\dprod{i=1}{n} T_{q,x}^{\nu_i}.  
\end{align}
It is known by \cite{NS2012}
that these operators satisfy the {\em Wronski relations} 
\begin{align}
\dsum{i+j=k}{}(-1)^i(1-t^iq^j)D_i^x H_j^x=0\quad(k=1,2,\ldots).  
\end{align}
From this fact, it turns out 
that each $H_l^x$ belong to the commutative 
ring $\mathbb{C}[D_1^x,\ldots,D_n^x]$ of Ruijsenaars-Macdonald operators. 
Hence we see that the operators $H_l^{x}$ 
commute with each other, and that they commute with 
the operators $D_r^x$ ($r=1,\ldots,n$):  For all $k,l=0,1,2,\ldots$, 
\begin{align}
H_k^{x}H_l^{x}=H_l^{x}H_k^{x},\quad
H_k^{x} D_l^{x} =D_l^{x} H_k^{x}.
\end{align} 
In terms of the two generating functions 
\begin{align}
D^x(u)=\dsum{r=0}{n}(-u)^r D^x_r, \quad\mbox{and}\quad
H^x(u)=\dsum{l=0}{\infty} u^l H_l^{x}
\end{align}
for Ruijsenaars-Macdonald operators of column type and of row type, 
the Wronski relations above can be written as
\begin{align}
D^x(u) H^x(u)=D^x(tu) H^x(qu).  
\end{align}
From this relation, it follows 
that if $f(x;s)$ satisfies the eigenfunction 
equation
\begin{align}\label{eq:eigenDfxs}
D^x(u)\,f(x;s)=f(x;s)\,\dprod{i=1}{n}(1-us_i),
\end{align}
then it satisfies 
\begin{align}
H^x(u)\,f(x;s)=f(x;s)\,\dprod{i=1}{n}\dfrac{(tus_i;q)_\infty}{(us_i;q)_\infty}.  
\end{align}
We remark that the operator $H^x(u)$ can be rewritten in the form 
\begin{align}\label{eq:Haltexp}
H^{x}(u)
&=
\dprod{1\le i<j\le n}{}
\dfrac{(qx_i/x_j;q)_\infty}{(qx_i/tx_j;q)_\infty}
\dsum{\nu\in\mathbb{N}}{}
u^{|\nu|}t^{\sum_i (i-1)\nu_i}
\dprod{i=1}{n}\dfrac{(t;q)_{\nu_i}}{(q;q)_{\nu_i}}
\nonumber\\
&\quad\cdot
\dprod{1\le i<j\le n}{}
\dfrac{(tx_i/x_j;q)_{\nu_i}}{(qx_i/x_j;q)_{\nu_i}}
\dfrac{(q^{-\nu_j+1}x_i/t x_j;q)_{\nu_i}}{(q^{-\nu_j}x_i/x_j;q)_{\nu_i}}\ 
T_{q,x}^{\nu}
\dprod{1\le i<j\le n}{}
\dfrac{(qx_i/tx_j;q)_\infty}{(qx_i/x_j;q)_\infty}.  
\end{align}

\par\medskip
For each $q$-difference operator
\begin{align}
A^x=A(x;T_{q.x})
=\dsum{\nu\in\mathbb{N};\,|\nu|\le m}\,A_{\nu}(x)T_{q,x}^\nu
\in \mathbb{C}(x)[T_{q,x}^{\pm 1}],
\end{align}
we consider the $q$-difference operator 
\begin{align}
A^{x^{-1}}=A(x^{-1},T_{q,x}^{-1})=
\dsum{\nu\in\mathbb{N};\,\nu|\le m}{}\,A_{\nu}(x^{-1})T_{q,x}^{-\nu}
\end{align}
obtained from $A^{x}$ by inverting the 
variables $x_i$ ($i=1,\ldots,n$).  
As for Ruijsenaars-Macdonald $q$-difference operators of column type, one has
\begin{align}
D^{x^{-1}}_r=t^{(n-1)r}D_{n-r}^x (D_n^{x})^{-1}\quad(r=0,1,\ldots,n).  
\end{align}
Hence we see the commutative ring 
$\mathbb{C}[D^x_1,\ldots,D^x_{n-1},(D^x_n)^{\pm1}]$ of
Ruijsenaars-Macdonald operators is stable 
by the operation $A^x \to A^{x^{-1}}$. 
By using this fact, one can show that the eigenfunction equation 
\eqref{eq:eigenDfxs}
implies
\begin{align}
D^{x^{-1}}\!(u)\,f(x;s)=f(x;s)\,\dprod{i=1}{n}(1-t^{n-1}u/s_i),
\end{align}
and hence
\begin{align}
H^{x^{-1}}\!(u)\,f(x;s)=f(x;s)\,\dprod{i=1}{n}
\dfrac{(t^nu/s_i;q)_\infty}{(t^{n-1}u/s_i;q)_\infty}, 
\end{align}
namely,
\begin{align}\label{eq:Hinv}
H^{x^{-1}}\!(u/t^{n-1})\,f(x;s)=f(x;s)\,\dprod{i=1}{n}
\dfrac{(tu/s_i;q)_\infty}{(u/s_i;q)_\infty}.  
\end{align}

As we have seen in Section 1, it is convenient to transform 
the eigenfunction equation for $f(x;s)$ 
into the equation for $\psi(x;s)$ defined by 
\begin{align}
f(x;s)=x^{\lambda}\dprod{1\le i<j\le n}{}
\dfrac{(qx_j/x_i;q)_\infty}{(qx_j/tx_i;q)_\infty}\,\psi(x;s)
\end{align}
with parameter $\lambda$ such that $s=t^\delta q^\lambda$.  
Then the equation \eqref{eq:Hinv} for $f(x;s)$ is rewritten as 
\begin{align}
K^{(x;s)}(u)\,\psi(x;s)
=\psi(x;s)\,\dprod{i=1}{n}\dfrac{(tu/s_i;q)_\infty}{(u/s_i;q)_\infty}, 
\end{align}
where 
\begin{align}
K^{(x;s)}(u)=\dprod{1\le i<j\le n}{}
\dfrac{(qx_j/tx_i;q)_\infty}{(qx_j/x_i;q)_\infty}
x^{-\lambda}H^{x^{-1}}\!(u/t^{n-1})x^{\lambda}
\dprod{1\le i<j\le n}{}
\dfrac{(qx_j/x_i;q)_\infty}{(qx_j/tx_i;q)_\infty}.
\end{align}
Form \eqref{eq:Haltexp}, this operator $K^{(x;s)}(u)$ is 
determined as follows:
\begin{align}\label{eq:defKop}
K^{(x;s)}(u)
=\dsum{\nu\in\mathbb{N}^n}{}
\dprod{i=1}{n}(u/s_i)^{\nu_i}
\dprod{i=1}{n}
\dfrac{(t;q)_{\nu_i}}{(q;q)_{\nu_i}}
\dprod{1\le i<j\le n}{}
\dfrac{(tx_j/x_i;q)_{\nu_i}}{(qx_j/x_i;q)_{\nu_i}}
\dfrac{(q^{-\nu_j+1}x_j/t x_i;q)_{\nu_i}}{(q^{-\nu_j}x_j/x_i;q)_{\nu_i}}\ 
T_{q,x}^{-\nu}. 
\end{align}
We also use the notation $K^{(x;s|q,t)}$ for this operator when we 
need to specify the parameter $t$.

\subsection{Recurrence by $q$-difference operators}

Let us consider the sequence of formal power series
\begin{align}\label{eq:defpsiwithp}
\psi_n(s;x|q,t)=
\dprod{1\le i<j\le n}{}
\dfrac{(qs_j/ts_i;q)_\infty}{(qs_j/s_i;q)_\infty}\,
p_n(s;x|q,t)\quad (n=1,2,\ldots)
\end{align}
so that 
\begin{align}
\varphi_n(s;x)=\dprod{1\le i<j\le n}{}
\dfrac{(qx_j/x_i;q)_\infty}{(qx_j/tx_i;q)_\infty}\
\dfrac{(qs_j/s_i;q)_\infty}{(qs_j/ts_i;q)_\infty}\,\psi_n(s;x|q,t).  
\end{align}
In Section 3, by the inductive construction by Jackson integrals, 
we proved $x^\lambda \varphi_n(s;x|q,t)$ 
with $s=t^{\delta}q^\lambda$ solves the eigenfunction equation
\eqref{eq:eigenDfxs} in $x$ variables for $n=1,2,\ldots$.  
Note that $\psi_n(s;x|q,t)$ has leading coefficient 1 both in 
$x$ variables and $s$ variables, and hence we already know 
that $\psi_n(s;x|q,t)=\psi_n(s;x|q,q/t)$ for all $n=1,2,\ldots$.  

\par\medskip
An important observation is that the recurrence relation
for $\psi_n(s;x|q,t)$
\begin{align}\label{eq:indformpsi3}
\psi_{n+1}(s;x|q,t)
&=
\dprod{i=1}{n}
\dfrac{(qs_{n+1}/ts_i;q)_{\infty}}{(qs_{n+1}/s_i;q)_{\infty}}
\dsum{\nu\in \mathbb{N}^n}{}
\dprod{i=1}{n}
\dfrac{(t;q)_{\nu_i}}{(q;q)_{\nu_i}}
\dprod{1\le i<j\le n+1}{}
\dfrac{(tx_j/x_i;q)_{\nu_i}}{(qx_j/x_i;q)_{\nu_i}}
\nonumber\\
&
\quad\cdot\dprod{1\le i<j\le n}{}
\dfrac{(q^{-\nu_j+1}x_j/tx_i;q)_{\nu_i}}{(q^{-\nu_j}x_j/x_i;q)_{\nu_i}}
\dprod{i=1}{n}(qs_{n+1}/ts_i)^{\nu_i}
\psi_n(s;q^{-\nu}x|q,t)
\end{align}
can be described in terms of the $q$-difference operator $K^{(x;s|q,t)}(u)$
defined by \eqref{eq:defKop} above.  
Note that this recurrence formula follows directly from that of $p_n(x;s|q,t)$ 
in \eqref{eq:recpn} by the definition \eqref{eq:defpsiwithp}. 
Also, we remark that this formula with $t$ and $q/t$ exchanged
arose naturally in the previous section from the Jackson integral 
representation. 
In fact, the recurrence relation \eqref{eq:indformpsi3} can be 
expressed as 
\begin{align}\label{eq:qDpsi2}
\psi_{n+1}(s;x|q,t)&=
\dprod{i=1}{n}
\dfrac{(tx_{n+1}/x_i;q)_\infty}{(qx_{n+1}/x_i;q)_\infty}
\dfrac{(qs_{n+1}/ts_i;q)_\infty}{(qs_{n+1}/s_i;q)_\infty}
\nonumber\\
&\quad\cdot
K^{(x;s|q,t)}(qs_{n+1}/t)\,
\dprod{i=1}{n}
\dfrac{(qx_{n+1}/x_i;q)_\infty}{(tx_{n+1}/x_i;q)_\infty}
\psi_n(s;x|q,t).   
\end{align}

By using this expression, we can show the duality relation 
$\psi_n(x;s|q,t)=\psi_n(s;x|q,t)$ for $\psi_n(s;x|q,t)$ ($n=1,2,\ldots$) 
by the induction on $n$ starting from 
$\psi_1(s;x|q,t)=1$.  
We suppose that $\psi_n(x;s|q,t)=\psi_n(s;x|q,t)$ as the induction 
hypothesis.   Then we know that 
$e_n(x;s)\varphi_n(s;x)$ is a formal solution of the eigenfunction equation 
is $s$ variables, and hence
\begin{align}
K^{(s;x|q,t)}(u)\,\psi(s;x|q,t)=\psi(s;x|q,t)\,
\dprod{i=1}{n}
\dfrac{(tu/x_i;q)_\infty}{(u/x_i;q)_\infty},
\end{align}
as well as
\begin{align}\label{eq:Keigen2}
K^{(s;x|q,q/t)}(u)\,\psi(s;x|q,t)=\psi(s;x|q,t)\,
\dprod{i=1}{n}
\dfrac{(qu/tx_i;q)_\infty}{(u/x_i;q)_\infty}
\end{align}
by the symmetry $\psi(s;x|q,t)=\psi(s;x|q,q/t)$.  
By formula \eqref{eq:Keigen2} with $u=tx_{n+1}$, 
the recurrence formula \eqref{eq:qDpsi2}
is rewritten in the form
\begin{align}\label{eq:recKK}
\psi_{n+1}(s;x|q,t)&=
\dprod{i=1}{n}
\dfrac{(tx_{n+1}/x_i;q)_\infty}{(qx_{n+1}/x_i;q)_\infty}
\dfrac{(qs_{n+1}/ts_i;q)_\infty}{(qs_{n+1}/s_i;q)_\infty}
\nonumber\\
&\quad\cdot
K^{(x;s|q,t)}(qs_{n+1}/t)\,
K^{(s;x|q,q/t)}(tx_{n+1})\,
\psi_n(s;x|q,t).  
\end{align}
By the symmetry $\psi_n(s;x|q,t)=\psi(x;s|q,t)=\psi(x;s|q,q/t)$, 
the right-hand side is symmetric with respect to exchanging 
$x\leftrightarrow s$ and $t\leftrightarrow q/t$ simultaneously.  
Hence we have
\begin{align}
\psi_{n+1}(s;x|q,t)=\psi_{n+1}(x;s|q,q/t)=\psi_{n+1}(x;s|q,t),
\end{align}
as desired, by using the symmetry of $\psi_{n+1}(s;x|q,t)$ 
with respect to $t\leftrightarrow q/t$. 

\par\medskip
The argument above, together with the recurrence by Jackson integrals 
in the previous section, provides a proof of duality and 
bispectrality of the formal power series
\begin{align}
f_n(x;s|q,t)=e_n(x;s)\varphi_n(x;s|q,t)\in e_n(x;s)\mathbb{C}[[x^{-Q_+}]][[s^{-Q_+}]]
\end{align}
which does not depend on the theory of Macdonald polynomials.
\begin{thm}  \label{Kxs}
The sequence of formal power series $\psi_{n}(s;x|q,t)$ $(n=1,2,\ldots)$ satisfies
the following recurrence relation by the $q$-difference operators\,$:$ 
\begin{align}
\psi_{n+1}(s;x|q,t)&=
\dprod{i=1}{n}
\dfrac{(tx_{n+1}/x_i;q)_\infty}{(qx_{n+1}/x_i;q)_\infty}
\dfrac{(qs_{n+1}/ts_i;q)_\infty}{(qs_{n+1}/s_i;q)_\infty}
\nonumber\\
&\quad\cdot
K^{(x;s|q,t)}(qs_{n+1}/t)\,
K^{(s;x|q,q/t)}(tx_{n+1})\,
\psi_n(s;x|q,t), 
\end{align}
where $K^{(x;s|q,t)}(u)$ is the $q$-difference operator defined by 
\eqref{eq:defKop}.  
\end{thm}
\qed

\begin{thm}  
The formal power series $\varphi_n(x;s|q,t)$ and $\psi_n(x;s|q,t)$ 
satisfy the duality relation 
\begin{align}
\varphi_n(x;s|q,t)=\varphi_n(s;x|q,t),\quad
\psi_n(x;s|q,t)=\psi_n(s;x|q,t),
\end{align}
for $n=1,2,\ldots$.  Hence 
\begin{align}
f_n(x;s|q,t)&=e_n(x;s)\,\varphi_n(x;s;q|q,t)
\nonumber\\
&=
e_n(x;s)
\dprod{1\le i<j\le n}{}
\dfrac{(qx_j/x_i;q)_\infty}{(qx_j/tx_i;q)_\infty}
\dfrac{(qs_j/s_i;q)_\infty}{(qs_j/ts_i;q)_\infty}
\,
\psi_n(x;s|q,t)
\end{align}
is a formal solution of the bispectral problem \eqref{eq:bispf} for 
Ruijsenaars-Macdonald operators.   
Furthermore, $\psi_n(x;s|q,t)$ is invariant 
under the change of parameters $t\leftrightarrow q/t$.  
\end{thm}
\qed

In explicit terms, 
the recurrence relation \eqref{eq:recKK} means
\begin{align}\label{eq:recKKexplicit}
&\psi_{n+1}(s;x|q,t)
\nonumber\\
&=
\dprod{i=1}{n}
\dfrac{(tx_{n+1}/x_i;q)_\infty}{(qx_{n+1}/x_i;q)_\infty}
\dfrac{(qs_{n+1}/ts_i;q)_\infty}{(qs_{n+1}/s_i;q)_\infty}
\nonumber\\
&\quad\cdot
\dsum{\mu,\nu\in \mathbb{N}^n}{}
\dprod{i=1}{n}
\dfrac{(t;q)_{\mu_i}}{(q;q)_{\mu_i}}
\dfrac{(q/t;q)_{\nu_i}}{(q;q)_{\nu_i}}
\nonumber\\
&\quad\cdot
\dprod{1\le i<j\le n}{}
\dfrac{(tx_j/x_i;q)_{\mu_i}}{(qx_j/x_i;q)_{\mu_i}}
\dfrac{(q^{-\mu_j}qx_j/tx_i;q)_{\mu_i}}{(q^{-\mu_j}x_j/x_i;q)_{\mu_i}}
\dprod{1\le i<j\le n}{}
\dfrac{(qs_j/ts_i;q)_{\nu_i}}{(qs_j/s_i;q)_{\nu_i}}
\dfrac{(q^{-\nu_j}ts_j/s_i;q)_{\nu_i}}{(q^{-\nu_j}s_j/s_i;q)_{\nu_i}}
\nonumber\\
&\quad\cdot
\dprod{i=1}{n}
(qs_{n+1}/ts_i)^{\mu_i}
(tx_{n+1}/x_i)^{\nu_i}\,
q^{\sum_{i=1}^{n}\mu_i\nu_i}\, 
\psi_n(q^{-\nu}s;q^{-\mu}x|q,t).  
\end{align}
This recurrence formula manifestly shows 
the symmetry of $\psi_n(s;x|q,t)$ with respect to exchanging 
$x\leftrightarrow s$ and $t\leftrightarrow q/t$ simultaneously.

\section{Convergence of formal solutions}

\subsection{Summary on formal solutions of the bispectral problem}
In previous sections, 
we investigated the joint bispectral problem
\begin{align}\label{eq:JBispP}
D^x(u)f(x;s)=f(x;s)\dprod{i=1}{n}(1-us_i),
\quad
D^s(u)f(x;s)=f(x;s)\dprod{i=1}{n}(1-ux_i), 
\end{align}
in variables $x=(x_1,\ldots,x_n)$ and $s=(s_1,\ldots,s_n)$, 
and constructed an explicit formal solution
\begin{align}
f_n(x;s|q,t)&=e_n(x;s)\,\varphi_n(x;s|q,t),
\nonumber\\
\varphi_n(x;s)&=\dsum{\mu,\nu\in Q+}{}x^{-\mu}s^{-\nu}\,\varphi_{\mu,\nu}
\in \mathbb{C}
[[x^{-Q_+}]][[s^{-Q_+}]]
\end{align}
of this bispectral problem. 
Here we assume that $e_n(x;s)$ is (possibly multi-valued) 
a meromorphic function on $\mathbb{T}^n_x\times \mathbb{T}_s^n$, 
satisfying the symmetry condition 
$e_n(x;s)=e_n(s;x)$ and the $q$-difference equations 
\begin{align}
T_{q,x_i}(e_n(x;s))=e_n(x;s)s_i/t^{n-i},\quad
T_{q,s_i}(e_n(x;s))=e_n(x;s)x_i/t^{n-i}\quad
(i=1,\ldots,n). 
\end{align}
We introduced another formal power series 
\begin{align}
\psi_n(x;s|q,t)=\dsum{\mu,\nu\in Q_+}{}
x^{-\mu}s^{-\nu}\,\psi_{\mu,\nu}
\in \mathbb{C}[[x^{-Q_+}]][[s^{-Q_+}]]
\end{align}
such that 
\begin{align}
\varphi_n(x;s|q,t)=
\dprod{1\le i<j\le n}{}
\dfrac{(qx_j/x_i;q)_\infty}{(qx_j/tx_i;q)_\infty}
\dfrac{(qs_j/s_i;q)_\infty}{(qs_j/ts_i;q)_\infty}\,
\psi_n(x;s|q,t).  
\end{align}
This formal power series $\psi_n(x;s|q,t)$ satisfies the 
initial conditions 
\begin{align}\label{eq:psiini}
\psi_n(x;s|q,t)\big|_{x_{i+1}/x_i=0\ (i=1,\ldots,n-1)}=1,
\quad
\psi_n(x;s|q,t)\big|_{s_{i+1}/s_i=0\ (i=1,\ldots,n-1)}=1, 
\end{align}
and remarkable symmetry relations
\begin{align}\label{eq:psisym}
\psi_n(x;s|q,t)=\psi_n(s;x|q,t),\quad
\psi_n(x;s|q,t)=\psi_n(x;s|q,q/t).
\end{align}
The basic object in our framework was the formal power series 
\begin{align}
p_n(x;s|q,t)=
\dsum{\theta\in M_n}{}c_n(\theta; s|q,t)\ 
\dprod{1\le i<j\le n}{}(x_j/x_i)^{\theta_{ij}}
\in \mathbb{C}(s^{-Q_+})[[x^{-Q_+}]],
\end{align}
where 
\begin{align}
c_n(\theta;s|q,t)
&=
\dprod{k=2}{n}
\dprod{1\le i<j\le k}{}
\dfrac{(q^{\sum_{a>k}(\theta_{i,a}-\theta_{j,a})}ts_j/s_i;q)_{\theta_{i,k}}}
{(q^{\sum_{a>k}(\theta_{i,a}-\theta_{j,a})}qs_j/s_i;q)_{\theta_{i,k}}}
\nonumber\\
&\quad\cdot
\dprod{k=2}{n}
\dprod{1\le i\le j<k}{}
\dfrac{(q^{-\theta_{j,k}+\sum_{a>k}(\theta_{i,a}-\theta_{j,a})}qs_j/ts_i;q)_{\theta_{i,k}}}
{(q^{-\theta_{j,k}+\sum_{a>k}(\theta_{i,a}-\theta_{j,a})}s_j/s_i;q)_{\theta_{i,k}}}
\quad(\theta\in M_n). 
\end{align}
By means of this $p_n(x;s|q,t)$, $\varphi_n(x;s|q,t)$ and $\psi_n(x;s|q,t)$ 
are expressed as 
\begin{align}
\varphi_n(x;s|q,t)
=\dprod{1\le i<j\le n}{} 
\dfrac{(qs_j/s_i;q)_\infty}{(qs_j/ts_i;q)_\infty}\,
p_n(x;s|q,t)
\end{align}
and
\begin{align}
\psi_n(x;s|q,t)
=\dprod{1\le i<j\le n}{} 
\dfrac{(qx_j/tx_i;q)_\infty}{(qx_j/x_i;q)_\infty}\,
p_n(x;s|q,t).  
\end{align}
By the symmetry \eqref{eq:psisym}, this formal power series 
$\psi_n(x;s|q,t)$ can be expressed in four ways:
\begin{align}
\psi_n(x;s|q,t)
&
=
\dprod{1\le i<j\le n}{} 
\dfrac{(qx_j/tx_i;q)_\infty}{(qx_j/x_i;q)_\infty}\,
p_n(x;s|q,t)
=
\dprod{1\le i<j\le n}{} 
\dfrac{(qs_j/ts_i;q)_\infty}{(qs_j/s_i;q)_\infty}\,
p_n(s;x|q,t) 
\nonumber\\
&
=\dprod{1\le i<j\le n}{} 
\dfrac{(tx_j/x_i;q)_\infty}{(qx_j/x_i;q)_\infty}\,
p_n(x;s|q,q/t)
=
\dprod{1\le i<j\le n}{} 
\dfrac{(ts_j/s_i;q)_\infty}{(qs_j/s_i;q)_\infty}\,
p_n(s;x|q,q/t). 
\end{align}

\subsection{Convergence of formal solutions}

In view of symmetry, we will mainly use below the power series 
$\psi_n(x;s|q,t)$ so that 
\begin{align}
f_n(x;s|q,t)
&=e_n(x;s)\,\varphi_n(x;s|q,t)
\nonumber\\
&=
e_n(x;s)
\dprod{1\le i<j\le n}{}
\dfrac{(qx_j/x_i;q)_\infty}{(qx_j/tx_i;q)_\infty}
\dfrac{(qs_j/s_i;q)_\infty}{(qs_j/ts_i;q)_\infty}\,
\psi_n(x;s|q,t). 
\end{align}
From the initial conditions \eqref{eq:psiini}, we already know
\begin{align}
\psi_n(x;s|q,t)\in
\mathbb{C}(s^{-Q_+})[[x^{-Q_+}]]\cap 
\mathbb{C}(x^{-Q_+})[[s^{-Q_+}]].  
\end{align}
As to the expansion
\begin{align}
\psi_n(x;s|q,t)=\dsum{\mu\in Q_+}{}x^{-\mu}\psi_\mu(s),
\quad \psi_{\mu}(s)\in \mathbb{C}(s^{-Q_+}), 
\end{align}
the 
$\psi_{\mu}(s)$ are rational functions 
in 
$(s_2/s_1,\ldots,s_n/s_{n-1})$, 
regular at $(s_2/s_1,\ldots,s_n/s_{n-1})=0$, 
and have at most simple poles along 
\begin{align}
s_j/s_i= q^{-k-1}\quad (1\le i<j\le n;\ k=0,1,2,\ldots).  
\end{align}
On the other hand, $\psi_n(x;s|q,t)$ is expressed as
\begin{align}
\psi(x;s|q,t)
&=
\dprod{1\le i<j\le n}{} 
\dfrac{(qx_j/tx_i;q)_\infty}{(qx_j/x_i;q)_\infty}\,
p_n(x;s|q,t)
\nonumber\\
&=
\dprod{1\le i<j\le n}{} 
\dfrac{(qx_j/tx_i;q)_\infty}{(qx_j/x_i;q)_\infty}\,
\left(
\dsum{\mu\in Q_+}{} x^{-\mu} p_\mu(s)
\right),
\end{align}
where 
\begin{align}
p_\mu(s)=\dsum{\theta\in M_n(\mu)}{} c_{n}(\theta;s|q,t).  
\end{align}
Note that these $c_n(\theta;s|q,t)$ may have multiple poles 
along 
\begin{align}
s_j/s_i=q^{k}\quad(1\le i<j\le n;\ k\in\mathbb{Z}). 
\end{align}
\par\medskip

We denote by $\mathbb{C}^{n-1}_z$ the 
$(n-1)$-dimensional affine space with canonical 
coordinates $z=(z_1,\ldots,z_{n-1})$,   
and define a holomorphic mapping 
$\pi: \mathbb{T}^n_x\to\mathbb{C}^{n-1}_z$ 
by 
\begin{align}
\pi(a)=(a_2/a_1,\ldots,a_{n}/a_{n-1})\quad\mbox{for each}\quad
a=(a_1,\ldots,a_n)\in\mathbb{T}^n_x, 
\end{align}
so that $\pi^\ast(z_i)=x_{i+1}/x_i$ ($i=1,\ldots,n$), where 
$\pi^\ast:\ \mathcal{O}_{\mathbb{C}^{n-1}}\to 
\pi_\ast(\mathcal{O}_{\mathbb{T}^n})$
denotes the pull-back by $\pi$ 
in the sense of sheaves of holomorphic functions. 
Similarly, for the $s$ variables, we use the $(n-1)$-dimensional affine space 
$\mathbb{C}^{n-1}_w$ with 
canonical coordinates $w=(w_1,\ldots,w_{n-1})$ such that 
$\pi^\ast(w_i)=s_{i+1}/s_{i}$ ($i=1,\ldots,n-1$).  

In view of the singularity of $p_\mu(s)$, we define an open subset 
$D_w\subset\mathbb{C}^{n-1}_w$
by 
\begin{align}
D_w&=\pr{ w=(w_1,\ldots,w_{n-1})\in\mathbb{C}^{n-1}_w\ | \
w_i\cdots w_{j-1}\notin q^{-\mathbb{Z}}\cup\pr{0}\quad(1\le i<j\le n)}, 
\end{align}
so that 
\begin{align}
\pi^{-1}(D_w)&=\pr{ s=(s_1,\ldots,s_n)\in \mathbb{T}^n_{s}\ | \ 
s_j/s_i\notin q^{-\mathbb{Z}}\quad(1\le i<j\le n)}.  
\end{align}
For each $r>0$ we set 
\begin{align}
U_z(r)&=\pr{ z=(z_1,\ldots,z_{n-1})\in \mathbb{C}^{n-1}_z\ |\ 
|z_i|<r\ \ (i=1,\ldots,n-1)}, 
\nonumber\\
B_z(r)&=\pr{ z=(z_1,\ldots,z_{n-1})\in \mathbb{C}^{n-1}_z\ |\ 
|z_i|\le r\ \ (i=1,\ldots,n-1)}, 
\end{align}
so that 
\begin{align}
\pi^{-1}(B_z(r))&=\pr{ x=(x_1,\ldots,x_n)\in \mathbb{T}^n_{x}\ | \ 
|x_j/x_i|\le r^{j-i} \quad(1\le i<j\le n)}.  
\end{align}

\begin{prop}\label{prop:convpn}
For $n=2,3,\ldots$, 
we regard $p_n(x;s|q,t)$ as a formal power series in 
$z=(z_1,\ldots,z_{n-1})$ with coefficients in $\mathcal{O}(D_w)$$:$ 
\begin{align}\label{eq:fnseries}
p_n(x;s|q,t)=\dsum{\theta\in M_n}{}
c_n(\theta;s|q,t) \dprod{1\le i<j\le n}{}(x_j/x_i)^{\theta_{i,j}}
\in \mathcal{O}(D_w)[[z]]
\end{align}
We set 
$r_0=|q/t|^{\frac{n-2}{n-1}}$ if $|q/t|\le 1$, and 
$r_0=|t/q|$ if $|q/t|\ge1$.  
Then for any compact subset $K\subset D_w$ and for any 
$r<r_0$, 
this series \eqref{eq:fnseries} 
is absolutely convergent, uniformly on 
$B_z(r)\times K$.
Hence $p_n(x;s|q,t)$ defines a holomorphic function on 
$U_z(r_0)\times D_w$.
\end{prop}

Note that $c_n(\theta;s|q,t)$ is a product of factors of the form
\begin{align}
\dfrac{(q^l au;q)_{k}}{(q^l u;q)_k}\quad (k\in\mathbb{N},\ l\in\mathbb{Z}), 
\end{align}
with $(u,a)=(qs_j/s_i,t/q)$ or $(u,a)=(s_j/s_i,q/t)$.   

\begin{lem}\label{lem:estimate}
Let $a\in\mathbb{C}^\ast$ be a nonzero constant.  
For any compact subset 
$L\subset\mathbb{C}^\ast\backslash q^{\mathbb{Z}}$, 
there exists a positive constant $C_L>0$ such that,
for any finite subset $I\subset\mathbb{Z}$, 
\begin{align}
\left|
\dprod{i\in I}{}
\dfrac{1-q^i au}{1-q^i u}
\right|
\le C_L \max\pr{|a|,1}^{|I|}
\quad (u\in L).  
\end{align}
In particular, 
\begin{align}
\left|
\dfrac{(q^l au;q)_k}{(q^l u;q)_k}
\right|
\le C_L \max\pr{|a|,1}^k\quad(u\in L)
\end{align}
for any $k\in\mathbb{N}$ and $l\in\mathbb{Z}$. 
\end{lem}
\proof 
We take nonnegative integers $M,N\in\mathbb{N}$ such that
$|q|^{M+1}<|u|<|q|^{-N}$ for all $u\in L$.  
Given a finite subset $I\subset \mathbb{Z}$, we divide $I$ 
into three parts 
$I=I_{-}\sqcup I_0\sqcup I_+$ by setting 
\begin{align}
I_-=I\cap (-\infty,-M-1],\quad
I_0=I\cap [-M, N-1],\quad
I_+=I\cap [N,+\infty).
\end{align}
Suppose $u\in L$.  
Since $|q^iu|<1$ for all $i\ge N$, we have 
\begin{align}
\left|
\dprod{i\in I_{+}}{}
\dfrac{1-q^i au}{1-q^i u}
\right|
\le 
\dprod{i\in I_{+}}{}
\dfrac{1+|q^i au|}{1-|q^i u|}
\le
\dfrac{(-|q^Nau|;|q|)_\infty}{(|q^Nu|;|q|)_\infty}.  
\end{align}
Since $|q^{-i}/u|<1$ for $i\le -M-1$, 
\begin{align}
\left|
\dprod{i\in I_{-}}{}
\dfrac{1-q^i au}{1-q^i u}
\right|
=
|a|^{|I_{-}|}
\dprod{i\in I_{-}}{}
\dfrac{1+|q^{-i}/au|}{1-|q^{-i}/u|}
\le 
|a|^{|I_{-}|}
\dfrac{(-|q^{M+1}/au|;|q|)_\infty}{(|q^{M+1}/u|;|q|)_\infty}. 
\end{align}
Hence by using three continuous functions 
\begin{align}
c_+(u)=\dfrac{(-|q^Nau|;|q|)_\infty}{(|q^Nu|;|q|)_\infty}, \quad
c_-(u)=\dfrac{(-|q^{M+1}/au|;|q|)_\infty}{(|q^{M+1}/u|;|q|)_\infty}, 
\end{align}
and
\begin{align}
c_0(u)=
\max_{J\subseteq[-M,N-1]}\,
\left|
\dprod{i\in J}{}
\dfrac{1-q^i au}{1-q^i u}
\right|,
\end{align}
we estimate 
\begin{align}
\left|
\dprod{i\in I}{}
\dfrac{1-q^i au}{1-q^i u}
\right|
\le |a|^{|I_{-}|} c_{-}(u) c_0(u) c_+(u). 
\end{align}
Since $|a|^{|I_{-}|}\le |a|^{|I|}$ if $|a|\ge 1$,  
we obtain the estimate of Lemma by taking for $C_L$ 
the maximum of the continuous function 
$c_{-}(u)c_0(u)c_+(u)$ on $L$. 
\qed

\proof[Proof of Proposition \ref{prop:convpn}]\ 
Let $K$ be any compact subset of $D_w$.  
By the definition \eqref{eq:defcn}, 
\begin{align}
c_n(\theta;s|q,t)
&=
\dprod{1\le i<k\le n}{}
(q/t)^{\theta_{i,k}}
\dfrac{(t;q)_{\theta_{i,k}}}{(q;q)_{\theta_{i,k}}}
\nonumber\\
&\quad\cdot
\dprod{1\le i<j\le k\le n}{}
\dfrac{(q^{\sum_{a>k}(\theta_{i,a}-\theta_{j,a})}ts_j/s_i;q)_{\theta_{i,k}}}
{(q^{\sum_{a>k}(\theta_{i,a}-\theta_{j,a})}qs_j/s_i;q)_{\theta_{i,k}}}
\nonumber\\
&\quad\cdot
\dprod{1\le i<j<k\le n}{}
\dfrac{(q^{-\theta_{j,k}+\sum_{a>k}(\theta_{i,a}-\theta_{j,a})}qs_j/ts_i;q)_{\theta_{i,k}}}
{(q^{-\theta_{j,k}+\sum_{a>k}(\theta_{i,a}-\theta_{j,a})}s_j/s_i;q)_{\theta_{i,k}}}
\quad(\theta\in M_n). 
\end{align}
Suppose that $|q/t|\le 1$.  
Then by Lemma \ref{lem:estimate}
we see there exists a positive constant $C_K$ such that 
\begin{align}
|c_n(\theta;s|q,t)|\le C_K
\dprod{1\le i<j\le n}{}
\dfrac{(-|t|;|q|)_{\theta_{i,j}}}{(|q|;|q|)_{\theta_{i.j}}}
\dprod{1\le i<j\le n}{}
|t/q|^{(j-i-1)\theta_{i,j}}
\quad(s\in \pi^{-1}(K)). 
\end{align}
Hence
\begin{align}
\dsum{\theta\in M_n}{}|c_n(\theta;q,t)| \dprod{1\le i<j\le n}{}
|x_j/x_i|^{\theta_{ij}}
&\le 
C_K
\dsum{\theta\in M_n}{}
\dprod{1\le i<j\le n}{}\dfrac{(-|t|;|q|)_{\theta_{i,j}}}{(|q|;|q|)_{\theta_{i.j}}}
|(t/q)^{j-i-1}x_j/x_i|^{\theta_{ij}}
\nonumber\\
&
=
C_K
\dprod{1\le i<j\le n}{}
\dfrac{(-|(t/q)^{j-i-1}tx_j/x_i|;|q|)_\infty}
{(|(t/q)^{j-i-1}x_j/x_i|;|q|)_\infty}
\end{align}
if $|x_j/x_i|<|q/t|^{j-i-1}$ for all $1\le i<j\le n$.  
Hence for any $r>0$ such that $r^{k}<|q/t|^{k-1}$ 
for $k=1,\ldots,n-1$, 
the function series $p_n(x;s|q,t)$ is absolutely convergent, 
uniformly on $B_z(r)\times K$.  
The condition above for $r$ is equivalent to $r<r_0$, 
$r_0=|q/t|^{\frac{n-2}{n-1}}$.  
The case $|q/t|\ge 1$ can be treated in a similar way. 
\qed

\par\medskip
By Proposition \ref{prop:convpn}, the function
\begin{align}
\psi(x;s|q,t)
&=
\dprod{1\le i<j\le n}{} 
\dfrac{(qx_j/tx_i;q)_\infty}{(qx_j/x_i;q)_\infty}\,
p_n(x;s|q,t) 
\end{align}
is holomorphic on $U_z(r_0)\times D_w$. 
Hence, the Taylor expansion of $\psi_n(x;s|q,t)$ 
\begin{align}
\psi_n(x;s|q,t)=\dsum{\mu\in Q_+}{} x^{-\mu}\,\psi_{\mu}(s)
\in \mathcal{O}(D_w)[[z]]
\end{align}
in $z$ variables is normally convergent in $U_z(r_0)\times D_w$, 
namely, 
absolutely convergent, uniformly 
on any compact subset of $U_z(r_0)\times D_w$.  

We already know that the expansion coefficients $\psi_{\mu}(s)$ 
$(\mu\in Q_+)$ are rational functions in $w=(s_2/s_1,\ldots,s_n/s_{n-1})$, 
regular at $w=0$, and have at most simple poles along 
$s_j/s_i=w_i\cdots w_{j-1}=q^{-k-1}$ ($1\le i<j\le n;\ k=0,1,2,\ldots$).  
In order to eliminate these poles of the coefficients, 
we introduce the function 
\begin{align}\label{eq:defFn}
F_n(x;s|q,t)=\dprod{1\le i<j\le n}{}
(qx_j/x_i;q)_\infty(qs_j/s_i;q)_\infty\, \psi_n(x;s|q,t).  
\end{align}
As the formal power series in $z$ variables, 
the expansion coefficients are entire holomorphic functions 
in $s$ variables: 
\begin{align}\label{eq:TaylorF}
F_n(x;s|q,t)=\dsum{\mu\in Q_+}{}x^{-\mu} F_\mu(s)
\in \mathcal{O}(\mathbb{C}^{n-1}_w)[[z]].
\end{align}
On the other hand, this function 
$F_n(x;s|q,t)$ is holomorphic on $U_z(r_0)\times D_w$, 
and hence,
for any compact subset $K\subset D_w$, 
away from the divisors $s_j/s_i\in q^{\mathbb{Z}}$
($1\le i<j\le n$), 
and for any $r<r_0$,
this Taylor expansion 
is absolutely convergent, uniformly on $B_z(r)\times K$.  
Since the maximum of $|x^{-\mu}|$ on $B_z(r)$ 
is given by $r^{\br{\delta,\mu}}$, it means that 
\begin{align}
\dsum{\mu\in Q_+}{}|x^{-\mu}|\, |F_\mu(s)|
\le \dsum{\mu\in Q_+}{} r^{\br{\delta,\mu}}\norm{F_{\mu}}_K<+\infty
\end{align}
on $B_z(r)\times K$, 
where $\norm{F_{\mu}}_K$ stands for the supremum norm 
on the compact set $K$ with $F_\mu(s)$ regarded as a function in $w$
variables. 
\begin{lem}
For any compact subset $K\subset \mathbb{C}^{n-1}_w$, 
and for any $r<r_0$, 
the Taylor expansion \eqref{eq:TaylorF}
of $F_n(x;s|q,t)$ in $z$ variables 
is absolutely convergent, uniformly on $B_z(r)\times K$.  
Hence it defines a holomorphic function on 
$U_z(r_0)\times \mathbb{C}^{n-1}_w$.  
\end{lem}
\proof

Let $\rho>0$ be any irrational number, and set 
\begin{align}
K_\rho=\pr{ w=(w_1,\ldots,w_{n-1})\in\mathbb{C}^{n-1}_w\ 
|\ |w_i|=|q|^{-\rho}\  (i=1,\ldots,n-1)}
\end{align}
Then we have $K_\rho\subset D_w$, 
since $|s_j/s_i|=|z_i\cdots z_{j-1}|=|q|^{-(j-i)\rho}\notin |q|^{\mathbb{Z}}$ 
for any $i,j$ with $1\le i<j\le n$.  
Consider the corresponding closed polydisc in $\mathbb{C}^{n-1}_w$: 
\begin{align}
B_w(|q|^{-\rho})=\pr{ w=(w_1,\ldots,w_{n-1})\in\mathbb{C}^{n-1}_w\ 
|\ |w_i|\le |q|^{-\rho}\  (i=1,\ldots,n-1)}. 
\end{align}
Then for each holomorphic function 
$F_{\mu}(s)\in\mathcal{O}(C^{n-1}_w)$ ($\mu\in Q_+$), 
the maximum of its absolute values on $B_w(|q|^{-\rho})$ 
is attained on the Silov boundary $K_\rho$, namely, 
$\norm{F_\mu}_{B_w(|q|^{-\rho})}=\norm{F_\mu}_{K_\rho}$. 
This implies that, for any $r<r_0$, 
\begin{align}
\dsum{\mu\in Q_+}{} |x|^{-\mu} |F_\mu(s)|
\le \dsum{\mu\in Q_+}{} r^{\br{\delta,\mu}}\norm{F_\mu}_{K_\rho}<+\infty, 
\end{align}
uniformly on $B_z(r)\times B_w(|q|^{-\rho})$.  
Since $\rho>0$ can be taken arbitrarily large, 
the Taylor expansion \eqref{eq:TaylorF} of $F_n(x;s|q,t)$ 
in $z$ variables is absolutely convergent, on any compact 
subset in $U_z(r_0)\times \mathbb{C}^{n-1}.$\qed

\begin{prop}
As to the function $F_n(x;s|q,t)$ defined by \eqref{eq:defFn}, 
its Taylor series 
\begin{align}
F_{n}(x;s|q,t)=\dsum{\mu,\nu\in Q_+}{}x^{-\mu}s^{-\nu}\,F_{\mu,\nu}
\in\mathbb{C}[[x^{-Q_+}]][[s^{-Q_+}]]
\end{align}
is normally convergent in 
$\mathbb{C}^{n-1}_z\times\mathbb{C}^{n-1}_w$.  
Hence, $F_n(x;s|q,t)$
is continued to an entire holomorphic function on 
$\mathbb{C}^{n-1}_z\times \mathbb{C}^{n-1}_w$.  
\end{prop}
\proof
We know that this Taylor series is normally convergent in 
$U_z(r_0)\times \mathbb{C}^{n-1}_w$.   From the 
symmetry $F_n(x;s|q,t)=F_{n}(s;x|q,t)$, it is also normally 
convergent in $\mathbb{C}^{n-1}_z\times U_w(r_0)$.  
The domain of convergence of this power series
contains both 
$U_z(r_0)\times \mathbb{C}^{n-1}_w$ 
and $\mathbb{C}^{n-1}_z\times U_w(r_0)$, 
and hence it must be the whole affine space 
$\mathbb{C}^{n-1}_z\times \mathbb{C}^{n-1}_w$ 
by the logarithmic convexity of the domain of 
convergence. 
In fact, consider the compact subset $B_z(a)\times B_w(b)$ 
for arbitrary $a,b>0$.  Take sufficiently large 
$c>0$ so that $a/c<r_0$ and $b/c<r_0$. 
Then we have
\begin{align}
a^{\br{\delta,\mu}}b^{\br{\delta,\nu}}
\le\frac{1}{2}\left(
(a/c)^{\br{\delta,\mu}}(bc)^{\br{\delta,\nu}}
+
(ac)^{\br{\delta,\mu}}(b/c)^{\br{\delta,\nu}}
\right)
\quad(\mu,\nu\in Q_+), 
\end{align}
and hence 
\begin{align}
&\dsum{\mu,\nu\in Q_+}{}|x^{-\mu}|\,|s^{-\nu}|\,|F_{\mu,\nu}|
\le 
\dsum{\mu,\nu\in Q_+}{}\,a^{\br{\delta,\mu}}b^{\br{\delta,\nu}}
|F_{\mu,\nu}|
\nonumber\\
&
\le\frac{1}{2}\bigg(
\dsum{\mu,\nu\in Q_+}{}
(a/c)^{\br{\delta,\mu}}(bc)^{\br{\delta,\nu}}|F_{\mu,\nu}|
+
\dsum{\mu,\nu\in Q_+}{}
(ac)^{\br{\delta,\mu}}(b/c)^{\br{\delta,\nu}}|F_{\mu,\nu}|
\bigg)
<+\infty,
\end{align}
on $B_z(a)\times B_w(b)$.  
\qed

\begin{thm}\label{holo}
For the formal solution 
\begin{align}
f_n(x;s|q,t)&=e_n(x;s)\,
\varphi_n(x;s|q,t),
\quad
\varphi_n(x;s|q,t)
\in \mathbb{C}[[x^{-Q_+}]][[s^{-Q_+}]]
\end{align}
of the bispectral problem \eqref{eq:JBispP}
for Ruijsenaars-Macdonald operators, 
introduce a formal power series $F_n(x;s|q,t)$
by setting 
\begin{align}
\varphi_n(x;s|q,t)
=\dfrac{F_n(x;s|q,t)}{\prod_{1\le i<j\le n}{}(qx_j/tx_i;q)_\infty (qs_j/ts_i;q)_\infty }. 
\end{align}
Then, $F_n(x;s|q,t)$ represents a holomorphic function on 
$\mathbb{C}^{n-1}_z\times \mathbb{C}^{n-1}_w$ 
in the variable $(z;w)=(z_1,\ldots,z_{n-1};w_1,\ldots,w_{n-1})$ 
with $z_i=x_{i+1}/x_i$, $w_i=s_{i+1}/s_i$ $(i=1,\ldots,n-1)$, 
depending holomorphically on $t\in\mathbb{C}^\ast$. 
Furthermore, it satisfies the symmetry conditions 
\begin{align}
F_n(x;s|q,t)=F_n(s;x|q,t),\quad F_n(x;s|q,t)=F_n(x;s|q,q/t).  
\end{align}
\end{thm}
\qed

For $n=2$, $F_2(x;s|q,t)$ is expressed as
\begin{align}
F_2(x;s|q,t)=
(t;q)_\infty(qx_2s_2/x_1s_1;q)_\infty\,
{}_2\phi_{1}\!
\left[\begin{matrix}
qx_2/tx_1, qs_2/ts_1 \\
qx_2s_2/x_1s_1
\end{matrix};\,
q,t
\right]
\quad(|t|<1)  
\end{align}
as we already remarked in Section 3, and by symmetry, 
\begin{align}
F_2(x;s|q,t)=
(q/t;q)_\infty(qx_2s_2/x_1s_1;q)_\infty\,
{}_2\phi_{1}\!
\left[\begin{matrix}
tx_2/x_1, ts_2/s_1 \\
qx_2s_2/x_1s_1
\end{matrix};\,
q,q/t
\right]
\quad(|q/t|<1). 
\end{align}

\subsection{Principal specialization}
We give a remark on the evaluation of $\varphi_n(x;s|q,t)$ at $x=t^\delta$. 
By duality, this question is equivalent to knowing the value at $s=t^\delta$.   
The recurrence formula \eqref{eq:recpn} for $s=t^\delta$ reduces to 
a single term with $\nu=0$ because of the existence of 
the factor
\begin{align}
\dprod{1\le i<j\le n+1}{}\dfrac{(ts_j/s_i;q)_{\nu_i}}{(qs_j/s_i;q)_{\nu_i}}, 
\end{align}
since $(ts_j/s_i;q)_{\nu_i}$ 
for $j=i+1$ vanishes unless $\nu_i=0$.  Hence $p_n(x;t^\delta|q,t)=1$ 
for all $n=1,2,\ldots$.  From 
\begin{align}
\varphi_n(x;t^\delta|q,t)=
\dprod{1\le i<j\le n}{}\dfrac{(qt^{i-j};q)_\infty}{(qt^{i-j-1};q)_\infty}
=\dprod{i=1}{n}\dfrac{(q/t;q)_\infty}{(q/t^i;q)_\infty}. 
\end{align}
we obtain 
\begin{align}
\varphi_n(t^\delta;s|q,t)
=\varphi_n(x;t^\delta|q,t)
=\dprod{i=1}{n}\dfrac{(q/t;q)_\infty}{(q/t^i;q)_\infty}.
\end{align}
According to the definition \eqref{eq:defphin}, this implies a nontrivial 
summation formula:
\begin{thm}\label{pric-sp}
Let $|t|>|q|^{-(n-2)}$.  
We have
\begin{align}
\dsum{\theta\in M_n}{}c_n(\theta;s|q,t)\,t^{\sum_{i<j}(i-j)\theta_{ij}}
=
\dprod{i=1}{n}\dfrac{(q/t;q)_\infty}{(q/t^i;q)_\infty}
\dprod{1\le i<j\le n}{}
\dfrac{(qs_j/ts_i;q)_\infty}{(qs_j/s_i;q)_\infty}.  
\end{align}
\end{thm}
\begin{proof}
By Proposition \ref{prop:convpn}, when $|t|>|q|^{-(n-2)}$, 
the left-hand side is normally convergent for 
$s_j/s_i\notin q^{\mathbb{Z}}$ ($1\le i<j\le n$). 
\end{proof}



\section{Case $n=3$}

\subsection{Results}

In this section, we present a direct method to check the structure of the divisor of the poles of $p_n(x;s|q,t)$ and the duality 
for the case $n=3$, by applying several transformation and summation formulas for basic hypergeometric series. 
Note that the duality \eqref{eq:psisym} can be stated as
\begin{align}
\prod_{1\leq i<j\leq n} {(q s_j/s_i;q)_\infty \over (q s_j/ts_i;q)_\infty }p_n(x;s|q,q/t)=
\prod_{1\leq i<j\leq n} {(q x_j/x_i;q)_\infty \over (t x_j/x_i;q)_\infty }p_n(s;x|q,t). \label{dual}
\end{align}

\begin{thm}\label{n=3-henkan}
We have
\begin{align}
p_3(x;s|q,q/t)&=
\sum_{\theta\in\mathsf{M}^{(3)}}
c_3(\theta;s_1,s_2,s_3|q,t)
(x_2/x_1)^{\theta_{1,2}}(x_3/x_1)^{\theta_{1,3}}(x_3/x_2)^{\theta_{2,3}}\nonumber\\
&=
\sum_{k=0}^\infty 
{(q/t;q)_k(q/t;q)_k(t;q)_k(t;q)_k \over 
(q;q)_k(qs_2/s_1;q)_k(qs_3/s_1;q)_k(qs_3/s_2;q)_k}
(q s_3/ts_1)^k(tx_3/x_1)^k\label{n=3}\\
&\times
\prod_{1\leq i<j\leq 3}
{}_2\phi_1\left[
{q^{k+1}/t,qs_j/ts_i \atop q^{k+1}s_j/s_i};q, t x_j/x_i\right].\nonumber
\end{align}
This manifestly shows that $p_3(x;s|q,q/t)$ has at most simple poles along the divisors $s_j/s_i=q^{-k-1}$ $(1\leq i<j\leq 3;k=0,1,2,\ldots)$.
\end{thm}

A proof  will be given in Section 7.3.
We can recast this into another form.
\begin{thm}\label{n=3}
We have 
\begin{align}
\varphi_3(x,s|q,t)&=
\prod_{1\leq i<j\leq 3}{(q s_j/s_i;q)_\infty \over (q s_j/ts_i;q)_\infty}
p_3(x;s|q,t)\nonumber\\
&=
\sum_{k\geq 0}
{(q/t;q)_k(q/t;q)_k \over 
(q;q)_k(t;q)_k} 
( qx_3s_3/x_1s_1)^k \\
&\times
\prod_{1\leq i<j\leq 3}
{(t;q)_\infty (q x_js_j/x_is_i;q)_\infty \over (q x_j/tx_i;q)_\infty (q s_j/ts_i;q)_\infty}\,\,
{}_2\phi_1\left[
{qx_j/tx_i,qs_j/ts_i \atop q x_js_j/x_is_i};q,q^k t\right] \qquad (|t|<1),\nonumber
\end{align}
which manifestly shows the duality $\varphi_3(x,s|q,t)=\varphi_3(s,x|q,t)$.
\end{thm}

\begin{proof}
We can proceed in a completely parallel manner as  (\ref{eq:n=2final}) for $n=2$, by using 
the $q$-binomial theorem and Theorem \ref{n=3-henkan}. 
\end{proof}

There is yet another way to see the duality manifestly. From Theorem \ref{n=3-henkan} and the $q$-binomial theorem, we have 
\begin{align}
&
\prod_{1\leq i<j\leq 3}{(q s_j/s_i;q)_\infty \over (q s_j/ts_i;q)_\infty}
p_3(x;s|q,q/t)
\nonumber\\
&=
\sum_{k\geq 0\atop \theta\in \mathsf{M}^{(3)}}
{(q/t;q)_k(q/t;q)_k(t;q)_k(t;q)_k \over 
(q;q)_k}
(q s_3/ts_1)^k(tx_3/x_1)^k\nonumber\\
&\times
\prod_{1\leq i<j\leq 3}
{(q^{k+1}/t;q)_{\theta_{i,j}}\over (q;q)_{\theta_{i,j}}}
{(q^{k+\theta_{i,j}+1} s_j/s_i;q)_\infty \over (q^{\theta_{i,j}+1} s_j/ts_i;q)_\infty}
(t x_j/x_i)^{\theta_{i,j}}\\
&=
\sum_{k\geq 0\atop \theta,\nu\in \mathsf{M}^{(3)}}
{(q/t;q)_k(q/t;q)_k(t;q)_k(t;q)_k \over 
(q;q)_k}
(q s_3/ts_1)^k(tx_3/x_1)^k\nonumber\\
&\times
\prod_{1\leq i<j\leq 3}
{(q^{k+1}/t;q)_{\theta_{i,j}}\over (q;q)_{\theta_{i,j}}}
{(q^{k}t;q)_{\nu_{i,j}}\over (q;q)_{\nu_{i,j}}}
(t x_j/x_i)^{\theta_{i,j}}(q s_j/t s_i)^{\nu_{i,j}}q^{\theta_{i,j}\nu_{i,j}},\nonumber
\end{align}
which indicates the duality of the form stated in (\ref{dual}).

\subsection{Some transformation formulas}
For our proof of Theorem \ref{n=3-henkan},
we prepare several transformation formulas.
We use the notations used in \cite{GR}.

\begin{prop}\label{koushiki-1}
Let 
$\theta,r$ be nonnegative integers. We have the transformation formula for 
very well-poised series
\begin{align}
&
{}_{r+9}W_{r+8}(a;q^{-\theta},q^{\theta}af,a_1,\cdots,a_r,
{\textstyle
\left(aq\over f\right)^{1\over 2}},
-{\textstyle
\left(aq\over f\right)^{1\over 2}},
{\textstyle
\left(aq^2\over f\right)^{1\over 2}},
-{\textstyle
\left(aq^2\over f\right)^{1\over 2}};q,z)\\
&=
{(aq,f^2/q;q)_{\theta} \over(af,f;q)_{\theta} }
 \sum_{m\geq 0}
{(q/f,q^{-\theta},aq/f;q)_m \over (q,q^{-\theta}q^2/f^2,aq;q)_m}
q^m 
{}_{r+5}W_{r+4}
(a;q^{-m},q^{m}aq/f,a_1,\cdots,a_r;q,z).\nonumber
\end{align}
\end{prop}

\begin{proof}
Note that from the $q$-Saalsch\"utz summation formula  \cite{GR}, we have
\begin{align}
&(f/q)^{k}{(af,f;q)_{\theta}\over (aq,f^2/q;q)_{\theta}}
{(q^{-\theta}q^2/f^2,q^{\theta}af;q)_k(aq;q)_{2k}\over 
 (q^{-\theta}q/f,q^{\theta}aq;q)_k(af;q)_{2k}}\nonumber\\
&=
{(f,q^{2k}af;q)_{\theta-k} \over (f^2/q,q^{2k}aq;q)_{\theta-k}}\\
&=
{}_3\phi_2\left[
{q^{2k} {aq/ f},q/f,q^{-\theta+k} \atop
q^{2k}aq,q^{-\theta+k}q^2/f^2};q,q\right].\nonumber
\end{align}
By using this, we can proceed as follows:
\begin{align}
&
{(af,f;q)_{\theta}\over 
(aq,f^2/q;q)_{\theta} }\times\mbox{LHS}\nonumber\\
&=
\sum_{k\geq 0}
{(af,f;q)_{\theta}\over 
(aq,f^2/q;q)_{\theta} }
{(a,q^{-\theta},q^{\theta}af;q)_k \over 
(q,q^{-\theta}q/f,q^{\theta}aq;q)_k }
{(a_1,\cdots,a_r;q)_k\over 
({\scriptstyle aq\over  a_1},\cdots,
{\scriptstyle aq\over  a_r};q)_k}
{(aq/f;q)_{2k} \over (af;q)_{2k}}{1-a q^{2k}\over 1-a} z^k\nonumber\\
&=
\sum_{k\geq 0}
{(f,q^{2k}af;q)_{\theta-k}\over 
(f^2/q,q^{2k}aq;q)_{\theta-k} }
{(a,q^{-\theta};q)_k \over 
(q,q^{-\theta}q^2/f^2;q)_k }
{(a_1,\cdots,a_r;q)_k\over 
({\scriptstyle aq\over  a_1},\cdots,
{\scriptstyle aq\over  a_r};q)_k}
{(aq/f;q)_{2k} \over (aq;q)_{2k}}{1-a q^{2k}\over 1-a} (qz/f)^k\nonumber\\
&=
\sum_{k\geq 0}\sum_{\ell\geq 0}
{(q^{2k}aq/f,q/f,q^{-\theta+k};q)_\ell \over 
(q^{2k}aq,q,q^{-\theta+k}q^2/f^2;q)_\ell }q^\ell
{(a,q^{-\theta};q)_k \over 
(q,q^{-\theta}q^2/f^2;q)_k }
{(a_1,\cdots,a_r;q)_k\over 
({\scriptstyle aq\over  a_1},\cdots,
{\scriptstyle aq\over  a_r};q)_k}\nonumber\\
&\qquad\times 
{(aq/f;q)_{2k} \over (aq;q)_{2k}}{1-a q^{2k}\over 1-a} (qz/f)^k\\
&=
\sum_{m\geq 0}\sum_{k=0}^m
{(aq/f,q/f,q^{-\theta};q)_m\over 
(aq,q,q^{-\theta}q^2/f^2;q)_m}q^m
{(a,q^{-m},q^m aq/f;q)_k \over 
(q,q^{-m}f,q^m aq;q)_k }
{(a_1,\cdots,a_r;q)_k\over 
({\scriptstyle aq\over  a_1},\cdots,
{\scriptstyle aq\over  a_r};q)_k}{1-a q^{2k}\over 1-a} z^k\nonumber\\
&=
\sum_{m\geq 0}
{(aq/f,q/f,q^{-\theta};q)_m\over 
(aq,q,q^{-\theta}q^2/f^2;q)_m}q^m
{}_{r+5}W_{r+4} (a;q^{-m},q^m aq/f,a_1,\cdots,a_r;q,z)\nonumber\\
&=
{(af,f;q)_{\theta}\over 
(aq,f^2/q;q)_{\theta} }\times\mbox{RHS}\nonumber.
\end{align}
\end{proof}

\begin{lem}\label{lem-2}
Let 
$\theta$ be a nonnegative integer. We have
\begin{align}
&\sum_{m\geq 0}
{(q/f,q^{-\theta},aq/f;q)_m\over 
(q,q^{-\theta}q^2/f^2,aq;q)_m}q^m
{}_{10}W_9 (a;b,c,d,e,f,g,q^{-m};q,a^3q^{m+3}/bcdefg)\nonumber
\\
&=
\sum_{j\geq 0}
\sum_{m\geq j}
{(q/f,q^{-\theta},aq/fg;q)_m \over  
(q,q^{-\theta}q^2/f^2,aq/g;q)_m }q^m\\
&\times
{(q^{-m},f,g,aq/de;q)_j \over 
(q,{ aq/ d},
{ aq/ e},
q^{-m}fg/a;q)_j }q^j 
{}_4\phi_3\left[
{q^{-j},d,e,{ aq/ bc} \atop
{ aq/ b},
{ aq/ c},
q^{-j}de/a};q,q\right].\nonumber
\end{align}
Note that this holds also when $b$ is depending on $m$. (In the proof of Theorem \ref{n=3-henkan},
we will consider that case $b=q^m a q/f$.)
\end{lem}

\begin{proof}
Recall Exercise 2.20 (p.53) in Gasper-Rahman \cite{GR}:
\begin{align}
&{}_{10}W_9(a;b,c,d,e,f,g,q^{-n};q,a^3q^{n+3}/bcdefg)\\
&=
{(aq,aq/fg;q)_n
\over (aq/f,aq/g;q)_n}
\sum_{j=0}^n
{(q^{-n},f,g,aq/de;q)_j\over (q,aq/d,aq/e,fgq^{-n}/a;q)_j}q^j
{}_4\phi_3
\left[ {q^{-j},d,e,aq/bc \atop aq/b,aq/c,deq^{-j}/a};q,q\right].\nonumber
\end{align}
Then we have
\begin{align}
&\mbox{LHS}\nonumber\\
&=\sum_{m\geq 0}\sum_{j\geq 0}
{(q/f,q^{-\theta_{1,2}},aq/f;q)_m\over 
(q,q^{-\theta_{1,2}}q^2/f^2,aq;q)_m}q^m\\
&\times
{(aq,aq/fg;q)_m\over (aq/f,aq/g;q)_m}
{(q^{-m},f,g,aq/de;q)_j\over 
(q,{aq/ d},
{ aq/ e},
q^{-m}fg/a;q)_j }q^j 
{}_4\phi_3\left[
{q^{-j},d,e,{ aq/bc} \atop
{ aq/ b},
{ aq/ c},
q^{-j}de/a};q,q\right]\nonumber\\
&=\mbox{RHS}.\nonumber
\end{align}
Here one should note that we have $(q^{-m};q)_j=0$ when $m<j$, and 
one can exchange the order of the summations as 
$\sum_{m\geq 0}\sum_{j\geq0}=
\sum_{j\geq 0}\sum_{m\geq0}=\sum_{j\geq 0}\sum_{m\geq j}$.
\end{proof}

\begin{prop}\label{koushiki-2}
Let $\theta$ be a nonnegative integer. We have
\begin{align}
&\sum_{m\geq 0}
{(q/f,q^{-\theta},aq/f;q)_m\over 
(q,q^{-\theta}q^2/f^2,aq;q)_m}q^m
{}_{10}W_9 (a;q^maq/f,c,d,e,f,af/e,q^{-m};q,aq^{2}/cdf)\nonumber\\
&=
{(e,f;q)_\theta\over (eq/f,f^2/q;q)_\theta}
\sum_{j\geq 0}
(aq^3/cdf^2)^j
{(c,d,f,af/e,q^{-\theta},q^{-\theta}f/e;q)_j\over
(q,aq/c,aq/d,aq/e,q^{-\theta+1}/e,q^{-\theta+1}/f;q)_j}\\
&\times
{}_5\phi_4\left[
{q^{-j},aq/cd,q/f,q^{\theta-j}f,q^{-j}e/a\atop
f,q^{-j+1}/d,q^{-j+1}/c, q^{\theta-j}eq/f};q,q\right].\nonumber
\end{align}
\end{prop}

\begin{proof}
In this proof, we use the notations $b=q^m aq/f,g=af/e$ for simplicity of display.
Use the Sears transformation \cite{GR} (2.10.4) twice, we have
\begin{align}
&
{}_4\phi_3\left[
{q^{-j},d,e,{\scriptstyle aq\over bc} \atop
{\scriptstyle aq\over b},
{\scriptstyle aq\over c},
q^{-j}de/a};q,q\right]\nonumber\\
&=
{(aq/be,q^{-j}d/a;q)_j \over (aq/b,q^{-j}de/a;q)_j}e^{j}
{}_4\phi_3\left[
{q^{-j},e,{\scriptstyle aq\over cd},b \atop
{\scriptstyle aq\over c},
q^{-j}{\scriptstyle be\over a},
{\scriptstyle aq\over d}};q,q\right]\\
&=
{(aq/be,q^{-j}d/a;q)_j \over (aq/b,q^{-j}de/a;q)_j}e^{j}
{(d,c;q)_j\over (aq/c,aq/d;q)_j}(aq/cd)^j
{}_4\phi_3\left[
{q^{-j},{\scriptstyle aq\over cd},q^{-j}{\scriptstyle b\over a},
q^{-j} {\scriptstyle e\over a}\atop
q^{-j}{\scriptstyle be\over a},
q^{-j}{\scriptstyle q\over d},
q^{-j}{\scriptstyle q\over c}};q,q\right].\nonumber
\end{align}
Use the $q$-Saalsch\"utz summation formula, then we have
\begin{align}
&
\sum_{m\geq j}
{(q/f,q^{-\theta},aq/fg;q)_m \over  
(q,q^{-\theta}q^2/f^2,aq/g;q)_m }q^m
{(q^{-m},aq/be;q)_j\over 
(q^{-m}fg/a,aq/b;q)_j}
{(q^{-j}b/a;q)_k \over (q^{-j}be/a;q)_k}
\nonumber\\
&=
(q/f^2)^j{(q^{-\theta};q)_j\over (q^{-\theta}q^2/f^2;q)_j}
{(q/f;q)_k\over (aq/g;q)_k}
\sum_{m\geq j}
{(q^{-\theta+j},q^k q/f,aq/fg;q)_{m-j} \over 
(q^{-\theta+j}q^2/f^2,q^kaq/g,q;q)_{m-j}}q^{m-j}
\nonumber\\
&=
(q/f^2)^j
{(q^{-\theta};q)_j\over (q^{-\theta}q^2/f^2;q)_j}
{(q/f;q)_k\over (aq/g;q)_k}
{(af/g,q^kf;q)_{\theta-j}\over (f^2/q,q^k aq/g;q)_{\theta-j}}.
\end{align}

Hence from these and Lemma \ref{lem-2}, we have
\begin{align}
&\mbox{LHS}\nonumber\\
&=
\sum_{j\geq 0}\sum_{k\geq 0}
\sum_{m\geq j}
{(q/f,q^{-\theta},aq/fg;q)_m \over  
(q,q^{-\theta}q^2/f^2,aq/g;q)_m }q^m
{(q^{-m},f,g,aq/de;q)_j \over 
(q,{ aq/ d},
{ aq/ e},
q^{-m}fg/a;q)_j }q^j \nonumber\\
&\times
{(aq/be,dq^{-j}/a;q)_j \over (aq/b,deq^{-j}/a;q)_j}e^{j}
{(d,c;q)_j\over (aq/c,aq/d;q)_j}(aq/cd)^j\nonumber\\
&\times
{(q^{-j},aq/cd,q^{-j}b/a,q^{-j}e/a;q)_k\over 
(q,q^{-j}be/a,q^{-j}q/d,q^{-j}q/c;q)_k}q^k\\
&=
\sum_{j\geq 0}\sum_{k\geq 0}
(q/f^2)^jq^je^{j}(aq/cd)^j
{(q^{-\theta};q)_j\over (q^{-\theta}q^2/f^2;q)_j}
{(q/f;q)_k\over (aq/g;q)_k}
{(af/g,q^kf;q)_{\theta-j}\over (f^2/q,q^k aq/g;q)_{\theta-j}}\nonumber\\
&\times
{(f,g,aq/de;q)_j \over 
(q,{ aq/ d},
{ aq/ e};q)_j } {(dq^{-j}/a;q)_j \over (deq^{-j}/a;q)_j}
{(d,c;q)_j\over (aq/c,aq/d;q)_j}\nonumber\\
&\times
{(q^{-j},aq/cd,q^{-j}e/a;q)_k\over 
(q,q^{-j}q/d,q^{-j}q/c;q)_k}q^k\nonumber\\
&=
{(f,af/g;q)_\theta\over (aq/g,f^2/q;q)_\theta}
\sum_{j\geq 0}
(aq^3/cdf^2)^j
{(c,d,f,g,q^{-\theta},q^{-\theta}g/a;q)_j\over
(q,aq/c,aq/d,aq/e,q^{-\theta}gq/af,q^{-\theta}q/f;q)_j}\nonumber\\
&\times
{}_5\phi_4\left[
{q^{-j},aq/cd,q/f,q^{\theta-j}f,q^{-j}e/a\atop
f,q^{-j+1}/d,q^{-j+1}/c, q^{\theta-j}aq/g};q,q\right]\nonumber\\
&=\mbox{RHS}.\nonumber
\end{align}
\end{proof}

\subsection{Proof of Theorem \ref{n=3-henkan}}\label{n=3-syoumei}
We specialize the variables $a,b,c,d,e,f,g$ as
\begin{align}
&a=q^{-\rho}s_2/s_1,\nonumber \\
&b=q^{-\rho+m}t s_2/s_1,\qquad \qquad\,\,\,\, aq/b=q^{1-m}/t,\nonumber\\
&c=q s_3/t s_1,\qquad  \qquad\qquad\quad  aq/c=q^{-\rho}ts_2/s_3,\label{kigou}\\
&d=q^{-\rho}s_2/s_3,\qquad \qquad\quad\,\,\,\,\,\,\, aq/d=qs_3/s_1,\nonumber\\
&e=q s_2/t s_1,\qquad \qquad\qquad\quad aq/e=q^{-\rho}t,\nonumber\\
&f=q/t,\qquad \qquad\qquad\qquad\,\,\,
 aq/f=q^{-\rho}ts_2/s_1,\nonumber\\
&g=q^{-\rho},\qquad \qquad\qquad\quad\,\,\,\, \,\,\,\, aq/g=qs_2/s_1,\nonumber
\end{align}
in the transformation formulas we have constructed above. 

\begin{lem}\label{lem-3}
Let $\theta,\rho$ be nonnegative integers. 
We have (with (\ref{kigou}))
\begin{align}
&\sum_{k\geq 0}c_3(\theta-k,k,\rho-k;s_1,s_2,s_3|q,q/t) \nonumber\\
&=
c_3(\theta,0,\rho;s_1,s_2,s_3|q,q/t) \\
&\times {}_{14}W_{13}(a;q^{-\theta},q^{\theta}af,c,d,e,f,g,
{\textstyle
\left(aq\over f\right)^{1\over 2}},
-{\textstyle
\left(aq\over f\right)^{1\over 2}},
{\textstyle
\left(aq^2\over f\right)^{1\over 2}},
-{\textstyle
\left(aq^2\over f\right)^{1\over 2}};q,t^2).\nonumber
\end{align}
\end{lem}

\begin{proof}
Straightforward calculation.
\end{proof}

Proposition \ref{koushiki-1}, \ref{koushiki-2} and Lemma \ref{lem-3} give us the following formula.

\begin{prop}\label{n=3-keisuu}
Let $\theta,\rho$ be nonnegative integers. We have
\begin{align}
&\sum_{k\geq 0}c_3(\theta-k,k,\rho-k;s_1,s_2,s_3|q,q/t)\nonumber\\
&=
t^\theta {(q/t,qs_2/ts_1;q)_\theta\over (q,qs_2/s_1;q)_\theta}
t^\rho {(q/t,qs_3/ts_2;q)_\rho\over (q,qs_3/s_2;q)_\rho}\label{keisuu}\\
&\times \sum_{j\geq 0}
{(q^{-\theta},q^{-\theta}s_1/s_2;q)_j\over 
(q^{-\theta}t,q^{-\theta}t s_1/s_2;q)_j}
{(q^{-\rho},q^{-\rho}s_2/s_3;q)_j\over 
(q^{-\rho}t,q^{-\rho}t s_2/s_3;q)_j}
{(q/t,qs_3/ts_1;q)_j\over (q,qs_3/s_1;q)_j}t^{3j}\nonumber\\
&\times
{}_5\phi_4\left[
{t,t,q^{-j},q^{\theta-j+1}/t,q^{\rho-j+1}/t\atop
q/t,q^{-j}ts_1/s_3,q^{\theta-j+1}s_2/s_1,q^{\rho-j+1} s_3/s_2};q,q\right].\nonumber
\end{align}
\end{prop}

Now we are ready to state our proof of Theorem \ref{n=3-henkan}.

\begin{proof}[Proof of Theorem \ref{n=3-henkan}]
The coefficient of the monomial $(x_2/x_1)^{\theta}(x_3/x_2)^{\rho}$ 
in the series $p_3(x;s|q,q/t)$ is $\sum_{k\geq 0}c_3(\theta-k,k,\rho-k;s_1,s_2,s_3|q,q/t)$. 
On the other hand, it can be easily shown that the 
 coefficient of the monomial $(x_2/x_1)^{\theta}(x_3/x_2)^{\rho}$ 
 in the series on RHS of (\ref{n=3}) is given by the RHS of (\ref{keisuu}).
Hence we have the equality  (\ref{n=3}).
\end{proof}
\section{Family of commutative integral transformations}
We revisit the problem considered in \cite{S2006}. 

Suppose  $g(y)$ is a joint eigenfunction of 
Ruijsenaars-Macdonald operators in $n$ variables $y=(y_1,\ldots,y_n)$ 
\begin{align}
D^y(u)g(y)=g(y)\,\dprod{i=1}{n}(1-ut^{n-i}q^{\lambda_i})
\end{align}
with complex parameters $\lambda=(\lambda_1,\ldots,\lambda_n)$.  
Define a function $f(x)$ in the same number of
variables $x=(x_1,\ldots,x_{n})$ by an integral 
transformation of the form 
\begin{align}
f(x)
=
\dint{}{}
\dprod{i,j=1}{n}
\dfrac{(tx_j/y_i;q)_\infty}{(x_j/y_i;q)_\infty}
\dprod{1\le i,j\le n;\,i\ne j}{} 
\dfrac{(y_i/y_j;q)_\infty}{(ty_i/y_j;q)_\infty}
c(x,y)g(y)\,d\omega(y),
\end{align}
where $d\omega(y)$ is a $q$-invariant measure on ${\mathbb T}^n_y$, 
$c(x,y)=c(x_1,\ldots,x_n,y_1,\ldots,y_n)$ satisfies the conditions 
\begin{align}
T_{q,x_i} c(x,y)=\alpha c(x,y),\qquad T_{q,y_i} c(x,y)=\alpha^{-1} c(x,y) \qquad (1\leq i\leq n),
\end{align}
where $\alpha\in {\mathbb C}^*$ is a parameter. 
One of the choice of such $ c(x,y)$ is
\begin{align}
c(x,y)=(x/y)^\lambda \prod_{i=1}^n {\theta(\alpha q y_i/s_i t x_i;q) \over \theta(q y_i/t x_i;q)}
\prod_{1\leq i<j\leq n} 
{\theta(q y_j/x_i;q)\theta(q y_j/ty_i;q) \over \theta(q y_j/tx_i;q)\theta(q y_j/y_i;q)},
\end{align}
where $s_i=t^{n-i}q^{\lambda_i}$ $(i=1,\ldots,n)$.
Then by using the same argument as in \cite{MN1997}, we can conclude that we have
\begin{align}
D^x(u)f(x)=f(x)\dprod{j=1}{n}(1-ut^{n-j}q^{\lambda_j}).
\end{align}

Let $t$ satisfies $|t|>1$, and fix $r$ such that $0<r<1$.
Consider a function $f(x)=x^\lambda \varphi(x)$ with $\varphi(x)$
being a holomorphic function in the variables $(x_2/x_1,\ldots,x_{n}/x_{n-1})$ defined on 
$ |x_{i+1}/x_i|<r$ $(i=1,\ldots,n-1)$.
Let  $\alpha\in {\mathbb C}^*$ be a parameter, and set 
\begin{align}
\psi(x)&=
\int_{C_1} \cdots \int_{C_n}\prod_{k=1}^n {dy_k\over 2\pi i y_k} 
\prod_{\ell=1}^n 
{\theta(\alpha q y_\ell/s_\ell tx_\ell;q)\over \theta(q y_\ell/tx_\ell;q)}
{(tx_\ell/y_\ell;q)_\infty \over(x_\ell/y_\ell;q)_\infty }\\
&\times
\prod_{1\leq i< j\leq n}{(tx_j/y_i;q)_\infty \over(x_j/y_i;q)_\infty }
{(qy_j/x_i;q)_\infty \over(qy_j/tx_i;q)_\infty }
(1-y_j/y_i){(qy_j/ty_i;q)_\infty \over(ty_j/y_i;q)_\infty }
\varphi(y),
\nonumber
\end{align}
where the contours $C_i$ are defined by $|y_i|=a_i$ with 
$
|q^{-1}t x_i|>a_i>|x_i|$ $(i=1,\ldots,n)$.
Then $\psi(x)$ is a holomorphic function  in the variables $(x_2/x_1,\ldots,x_{n}/x_{n-1})$ on 
$ |x_{i+1}/x_i|<|q /t|r$ $(i=1,\ldots,n-1)$.
Since $|t|>1$, our integration cycle is $q$-shift invariant.
We define the integral transform $(I(\alpha)  f)(x)$ of $f(x)=x^\lambda \varphi(x)$ by setting
$
\left(I(\alpha)  f\right)(x)= x^\lambda \psi(x)
$.

\begin{thm}
For any $\alpha,\beta\in {\mathbb C}^*$, we have
\begin{align}
&[D^x(u),I(\alpha)]=0,\label{D-I}\\
&[I(\alpha),I(\beta)]=0.\label{I-I}
\end{align}
\end{thm}

\begin{proof}
We have (\ref{D-I}) since 
$I(\alpha) ( x^\lambda p_n(x;s|q,t))$ is proportional to $ x^\lambda p_n(x;s|q,t)$.
Then (\ref{I-I}) follows from (\ref{D-I}).
\end{proof}






\begin{thebibliography}{99}
\bibitem{C2003}
I.~Cherednik:
Difference Macdonald-Mehta conjecture, 
Int. Math. Res. Not. IMRN 1997, 449--467.
\bibitem{C2009}
I.~Cherednik:
Whittaker limits of difference spherical functions,
Int. Math. Res. Not. IMRN 2009, 3793--3842. 

\bibitem{GR}
G. Gasper and M. Rahman,
Basic Hypergeometric Series, 
Cambridge University Press, Cambridge, (1990).


\bibitem{M1995}
I.G.~Macdonald:
{\em Symmetric Functions and Hall Polynomials}, Second Edition, 
Oxford Mathematical Monographs, Oxford University Press, 1995. 
\bibitem{M2003}
I.G.~Macdonald:
Affine Hecke algebras and orthogonal polynomials, 
Cambridge Tracts in Mathematics {\bf 153}, Cambridge University Press, 2003. 
\bibitem{MS2010}
M.~van Meer and J.~Stokman:
Double affine Hecke algebras and bispectral quantum 
Knizhnik-Zamolodchikov equations, 
Int. Math. Res. Not. IMRN 2010, 969--1040. 
\bibitem{MN1997}
K.~Mimachi and M.~Noumi:
An integral representation of eigenfunctions for Macdonald's $q$-difference operators,
T\^{o}hoku Math. J. {\bf 49} (1997), 517--525.
\bibitem{NS2012}
M.~Noumi and A.~Sano:
An infinite family of higher-order difference operators that 
commute with Ruijsenaars operators of type $A$, 
in preparation. 
\bibitem{R1987}
R.N.M.~Ruijsenaars:
Complete integrability of relativistic Calogero-Moser systems and 
elliptic function identities,  
Commun. Math. Phys. {\bf 110} (1987), 191--213.
\bibitem{S2005}
J.~Shiraishi:
A conjecture about raising operators for Macdonald polynomials, 
Lett. Math. Phys. {\bf 73} (2005), 71--81. 
\bibitem{S2006}
J.~Shiraishi:
A family of integral transformations and basic hypergeometric series, 
Commun. Math. Phys. {\bf 263} (2006), 439--460. 
\bibitem{S2010}
J.~Stokman:
The $c$-function expansion of a basic hypergeometric function associated to root systems,
arXiv:1109.0613.

\end{thebibliography}
\end{document}